\documentclass[a4paper]{amsart}
\usepackage{a4wide,amsfonts,amssymb,amsmath,amscd,color,tikz,tikz-3dplot,tikz-cd,todonotes}
\usepackage[english]{babel}
\allowdisplaybreaks

\newtheorem{lem}{Lemma}[section]

\newtheorem{theo}[lem]{Theorem}
\newtheorem{cor}[lem]{Corollary}
\newtheorem{rem}[lem]{Remark}

\def\bs{\boldsymbol}
\def\mr{\mathring}
\def\ol{\overline}
\def\To{\longrightarrow}

\def\incl{\hookrightarrow}
\DeclareMathOperator{\reals}{\mathbb{R}}
\DeclareMathOperator{\nat}{\mathbb{N}}

\def\ga{\Gamma}
\def\gat{\ga_{\!t}}
\def\gan{\ga_{\!n}}
\def\om{\Omega}

\def\rt{\reals^{3}}
\DeclareMathOperator{\RM}{\mathbb{RM}}

\DeclareMathOperator{\mcB}{\mathcal{B}}
\DeclareMathOperator{\mcH}{\mathcal{H}}
\DeclareMathOperator{\sfL}{\mathsf{L}}
\DeclareMathOperator{\sfH}{\mathsf{H}}
\DeclareMathOperator{\mbH}{\bs{\mathsf{H}}}
\DeclareMathOperator{\sfC}{\mathsf{C}}
\newcommand{\Harm}[2]{\mcH^{#1}_{\mathsf{#2}}}

\renewcommand{\L}[2]{\sfL^{#1}_{#2}}
\renewcommand{\H}[2]{\sfH^{#1}_{#2}}
\newcommand{\bH}[2]{\mbH^{#1}_{#2}}
\newcommand{\C}[2]{\sfC^{#1}_{#2}}
\newcommand{\B}[2]{\mcB^{#1}_{#2}}
\newcommand{\eps}{\varepsilon}
\DeclareMathOperator{\Lin}{Lin}
\DeclareMathOperator{\A}{A}
\DeclareMathOperator{\p}{\partial}
\DeclareMathOperator{\id}{id}
\DeclareMathOperator{\sym}{sym}
\DeclareMathOperator{\skw}{skw}
\DeclareMathOperator{\tr}{tr}
\DeclareMathOperator{\dev}{dev}
\DeclareMathOperator{\spn}{spn}
\DeclareMathOperator{\ed}{d}

\DeclareMathOperator{\grad}{grad}
\DeclareMathOperator{\rot}{rot}
\DeclareMathOperator{\divergence}{div}
\def\div{\divergence}
\DeclareMathOperator{\Grad}{Grad}
\DeclareMathOperator{\Rot}{Rot}
\DeclareMathOperator{\Div}{Div}
\DeclareMathOperator{\symGrad}{symGrad}

\newcommand{\symGradgat}{\symGrad_{\gat}}
\newcommand{\symGradgan}{\symGrad_{\gan}}
\newcommand{\symGradgatk}{\symGrad_{\gat}^{k}}
\newcommand{\symGradgank}{\symGrad_{\gan}^{k}}

\newcommand{\bsymGradgan}{\bs\symGrad_{\gan}}
\newcommand{\bsymGradgatk}{\bs\symGrad_{\gat}^{k}}
\newcommand{\bsymGradgank}{\bs\symGrad_{\gan}^{k}}
\newcommand{\rsymGradgat}{\mr\symGrad_{\gat}}

\newcommand{\symGradgats}{\symGrad_{\gat}^{*}}

\newcommand{\rsymGradgats}{\mr\symGrad{}_{\gat}^{*}}

\newcommand{\DivS}{\Div_{\S}}

\newcommand{\DivSgat}{\Div_{\S,\gat}}
\newcommand{\DivSgan}{\Div_{\S,\gan}}
\newcommand{\DivSgatk}{\Div_{\S,\gat}^{k}}
\newcommand{\DivSgank}{\Div_{\S,\gan}^{k}}
\newcommand{\DivSgatkmo}{\Div_{\S,\gat}^{k-1}}
\newcommand{\DivSgankmo}{\Div_{\S,\gan}^{k-1}}

\newcommand{\bDivSgan}{\bs\Div_{\S,\gan}}
\newcommand{\bDivSgatk}{\bs\Div_{\S,\gat}^{k}}
\newcommand{\bDivSgank}{\bs\Div_{\S,\gan}^{k}}
\newcommand{\bDivSgankmo}{\bs\Div_{\S,\gan}^{k-1}}
\newcommand{\rDivSgat}{\mr\Div_{\S,\gat}}

\newcommand{\DivSgats}{\Div_{\S,\gat}^{*}}

\newcommand{\rDivSgats}{\mr\Div{}_{\S,\gat}^{*}}

\DeclareMathOperator{\RotRot}{RotRot}
\newcommand{\RotRott}{\RotRot^{\top}}
\newcommand{\RotRottS}{\RotRot_{\S}^{\!\top}}

\newcommand{\RotRottSgat}{\RotRot_{\S,\gat}^{\!\top}}
\newcommand{\RotRottSgan}{\RotRot_{\S,\gan}^{\!\top}}
\newcommand{\RotRottSgatk}{\RotRot_{\S,\gat}^{\!\top\!,k}}
\newcommand{\RotRottSgatkmo}{\RotRot_{\S,\gat}^{\!\top\!,k-1}}

\newcommand{\RotRottSgank}{\RotRot_{\S,\gan}^{\!\top\!,k}}
\newcommand{\RotRottSgatkkmo}{\RotRot_{\S,\gat}^{\!\top\!,k,k-1}}
\newcommand{\RotRottSgankkmo}{\RotRot_{\S,\gan}^{\!\top\!,k,k-1}}
\newcommand{\RotRottSgatkpok}{\RotRot_{\S,\gat}^{\!\top\!,k+1,k}}
\newcommand{\RotRottSgankpok}{\RotRot_{\S,\gan}^{\!\top\!,k+1,k}}

\newcommand{\bRotRottSgan}{\bs\RotRot_{\S,\gan}^{\!\top}}
\newcommand{\bRotRottSgatk}{\bs\RotRot_{\S,\gat}^{\!\top\!,k}}
\newcommand{\bRotRottSgank}{\bs\RotRot_{\S,\gan}^{\!\top\!,k}}
\newcommand{\bRotRottSgatkkmo}{\bs\RotRot_{\S,\gat}^{\!\top\!,k,k-1}}
\newcommand{\bRotRottSgankkmo}{\bs\RotRot_{\S,\gan}^{\!\top\!,k,k-1}}
\newcommand{\rRotRottSgat}{\mr\RotRot{}_{\S,\gat}^{\!\top}}

\renewcommand{\S}{\mathbb{S}}

\newcommand{\PotP}{\mathcal{P}}
\newcommand{\PotQ}{\mathcal{Q}}
\newcommand{\PotN}{\mathcal{N}}
\newcommand{\norm}[1]{|#1|}

\newcommand{\scp}[2]{\langle#1,#2\rangle}

\title[Hilbert Complexes with Mixed Boundary Conditions -- Part 2: Elasticity Complex]
{Hilbert Complexes with Mixed Boundary Conditions\\
Part 2: Elasticity Complex}
\author{Dirk Pauly}
\author{Michael Schomburg}
\address{Fakult\"at f\"ur Mathematik, Universit\"at Duisburg-Essen, Germany}
\email[Dirk Pauly]{dirk.pauly@uni-due.de}
\email[Michael Schomburg]{michael.schomburg@uni-due.de}
\keywords{regular potentials, regular decompositions, compact embeddings,
Hilbert complexes, Mixed Boundary Conditions, elasticity complex}
\subjclass{}
\date{\today; {\it Corresponding Author}: Dirk Pauly}
\thanks{}

\begin{document}

\def\titlerepude{\sf Hilbert Complexes with Mixed Boundary Conditions:\\
Regular Decompositions, Compact Embeddings,\\ 
and Functional Analysis ToolBox\\
Part 2: Elasticity Complex}
\def\authorrepude{Dirk Pauly \& Michael Schomburg}
\def\daterepdue{\today}
\def\reportudemathyesno{no}
\def\reportudemathnumber{SM-UDE-827}
\def\reportudemathyear{2021}
\def\reportudematheingang{\daterepdue}
\newcommand{\preprintudemath}[5]{
\thispagestyle{empty}
\begin{center}\normalsize SCHRIFTENREIHE DER FAKULT\"AT F\"UR MATHEMATIK\end{center}
\vspace*{5mm}
\begin{center}#1\end{center}
\vspace*{5mm}
\begin{center}by\end{center}
\vspace*{0mm}
\begin{center}#2\end{center}
\vspace*{5mm}
\normalsize 
\begin{center}#3\hspace{69mm}#4\end{center}
\newpage
\thispagestyle{empty}
\vspace*{210mm}
Received: #5
\newpage
\addtocounter{page}{-2}
\normalsize
}
\ifthenelse{\equal{\reportudemathyesno}{yes}}
{\preprintudemath{\titlerepude}{\authorrepude}{\reportudemathnumber}{\reportudemathyear}{\reportudematheingang}}
{}


\begin{abstract}
We show that the elasticity Hilbert complex
with mixed boundary conditions on bounded strong Lipschitz domains
is closed and compact. The crucial results are compact embeddings 
which follow by abstract arguments using functional analysis
together with particular regular decompositions.
Higher Sobolev order results are proved as well.
This paper extends recent results on the de Rham Hilbert complex 
with mixed boundary conditions from \cite{PS2021a}
and recent results on the elasticity Hilbert complex 
with empty or full boundary conditions from \cite{PZ2020b}.
\end{abstract}


\maketitle
\setcounter{tocdepth}{3}
{
\tableofcontents}


\section{Introduction}
\label{sec:intro}%

In this paper we prove regular decompositions and resulting compact embeddings for 
the \emph{elasticity complex}
\begin{equation*}
\def\arrowlength{5ex}
\def\arrowdistance{0}
\begin{tikzcd}[column sep=\arrowlength]
\cdots 
\arrow[r, rightarrow, shift left=\arrowdistance, "\cdots"] 
& 
\L{2}{}(\om) 
\ar[r, rightarrow, shift left=\arrowdistance, "\symGrad"] 
& 
[2.5em]
\L{2}{\S}(\om)
\arrow[r, rightarrow, shift left=\arrowdistance, "\RotRottS"] 
& 
[2.5em]
\L{2}{\S}(\om)
\arrow[r, rightarrow, shift left=\arrowdistance, "\DivS"] 
& 
[1em]
\L{2}{}(\om)
\arrow[r, rightarrow, shift left=\arrowdistance, "\cdots"] 
&
\cdots.
\end{tikzcd}
\end{equation*}
This extends the corresponding results from \cite{PS2021a}
for the de Rham complex
\begin{equation*}
\def\arrowlength{5ex}
\def\arrowdistance{0}
\begin{tikzcd}[column sep=\arrowlength]
\cdots 
\arrow[r, rightarrow, shift left=\arrowdistance, "\cdots"] 
& 
\L{q-1,2}{}(\om) 
\ar[r, rightarrow, shift left=\arrowdistance, "\ed^{q-1}"] 
& 
[1em]
\L{q,2}{}(\om) 
\arrow[r, rightarrow, shift left=\arrowdistance, "\ed^{q}"] 
& 
\L{q+1,2}{}(\om) 
\arrow[r, rightarrow, shift left=\arrowdistance, "\cdots"] 
&
\cdots,
\end{tikzcd}
\end{equation*}
whose 3D-version for vector proxies is given by
\begin{equation*}
\def\arrowlength{5ex}
\def\arrowdistance{0}
\begin{tikzcd}[column sep=\arrowlength]
\cdots 
\arrow[r, rightarrow, shift left=\arrowdistance, "\cdots"] 
& 
\L{2}{}(\om) 
\ar[r, rightarrow, shift left=\arrowdistance, "\ed^{0}\widehat{=}\grad"] 
& 
[2.5em]
\L{2}{}(\om)
\arrow[r, rightarrow, shift left=\arrowdistance, "\ed^{1}\widehat{=}\rot"] 
& 
[2.5em]
\L{2}{}(\om)
\arrow[r, rightarrow, shift left=\arrowdistance, "\ed^{2}\widehat{=}\div"] 
& 
[2.5em]
\L{2}{}(\om)
\arrow[r, rightarrow, shift left=\arrowdistance, "\cdots"] 
&
\cdots.
\end{tikzcd}
\end{equation*}
We shall consider mixed boundary conditions
on a bounded strong Lipschitz domain $\om\subset\rt$. 

Like the de Rham complex, the elasticity complex has the geometric structure 
of a \emph{Hilbert complex}, i.e.,
\begin{equation*}
\def\arrowlength{5ex}
\def\arrowdistance{0}
\begin{tikzcd}[column sep=\arrowlength]
\cdots 
\arrow[r, rightarrow, shift left=\arrowdistance, "\cdots"] 
& 
\H{}{0} 
\ar[r, rightarrow, shift left=\arrowdistance, "\A_{0}"] 
& 
\H{}{1}
\arrow[r, rightarrow, shift left=\arrowdistance, "\A_{1}"] 
& 
\H{}{2}
\arrow[r, rightarrow, shift left=\arrowdistance, "\cdots"] 
&
\cdots,
\end{tikzcd}
\qquad
R(\A_{0})\subset N(\A_{1}),
\end{equation*}
where $\A_{0}$ and $\A_{1}$ are densely defined and closed (unbounded) linear operators
between Hilbert spaces $\H{}{\ell}$.
The corresponding \emph{domain Hilbert complex} is denoted by
\begin{equation*}
\def\arrowlength{5ex}
\def\arrowdistance{0}
\begin{tikzcd}[column sep=\arrowlength]
\cdots 
\arrow[r, rightarrow, shift left=\arrowdistance, "\cdots"] 
& 
D(\A_{0})
\ar[r, rightarrow, shift left=\arrowdistance, "\A_{0}"] 
& 
D(\A_{0})
\arrow[r, rightarrow, shift left=\arrowdistance, "\A_{1}"] 
& 
\H{}{2}
\arrow[r, rightarrow, shift left=\arrowdistance, "\cdots"] 
&
\cdots.
\end{tikzcd}
\end{equation*}

In fact, we show that the assumptions of \cite[Lemma 2.22]{PS2021a} hold,
which provides an elegant, abstract, and short way to prove the crucial compact embeddings
\begin{align}
\label{cptemb1}
D(\A_{1})\cap D(\A_{0}^{*})\incl\H{}{1}
\end{align}
for the elasticity Hilbert complex.
In principle, our general technique 
-- compact embeddings by regular decompositions and Rellich's selection theorem --
works for all Hilbert complexes known in the literature, see, e.g., 
\cite{AH2020a} for a comprehensive list of such Hilbert complexes.

Roughly speaking a regular decomposition has the form
$$D(\A_{1})=\H{+}{1}+\A_{0}\H{+}{0}$$
with regular subspaces $\H{+}{0}\subset D(\A_{0})$ and $\H{+}{1}\subset D(\A_{1})$
such that the embeddings $\H{+}{0}\incl\H{}{0}$ and $\H{+}{1}\incl\H{}{1}$
are compact. The compactness is typically and simply given by Rellich's selection theorem,
which justifies the notion ``regular''.
By applying $\A_{1}$ any regular decomposition implies regular potentials
$$R(\A_{1})=\A_{1}\H{+}{1}$$
by the complex property. 
The respective regular potential and decomposition operators 
\begin{align*}
\PotP_{\A_{1}}:R(\A_{1})&\to\H{+}{1},
&
\PotQ_{\A_{1}}^{1}:D(\A_{1})&\to\H{+}{1},
&
\PotQ_{\A_{1}}^{0}:D(\A_{1})&\to\H{+}{0}
\end{align*}
are bounded and satisfy $\A_{1}\PotP_{\A_{1}}=\id_{R(\A_{1})}$ as well as
$\id_{D(\A_{1})}=\PotQ_{\A_{1}}^{1}+\A_{0}\PotQ_{\A_{1}}^{0}$.

Note that \eqref{cptemb1} implies several important results related 
to the particular Hilbert complex by the so-called FA-ToolBox, 
such as closed ranges,
Friedrichs/Poincar\'e type estimates, Helmholtz typ decompositions,
and comprehensive solution theories,
cf.~\cite{P2017a,P2019b,P2019a,P2020a} and \cite{PZ2016a,PZ2020a,PZ2020b}.

For an historical overview on 
the compact embeddings \eqref{cptemb1} corresponding to the de Rham complex and Maxwell's equations,
i.e., Weck's or Weber-Weck-Picard's selection theorem,
see, e.g., the introductions in \cite{BPS2016a,NPW2015a},
the original papers \cite{W1974a,W1980a,P1984a,W1993a,J1997a,PWW2001a},
and the recent state of the art results 
for mixed boundary conditions and bounded weak Lipschitz domains in \cite{BPS2016a,BPS2018a,BPS2019a}.
Compact embeddings \eqref{cptemb1} corresponding to the biharmonic and the elasticity complex 
are given in \cite{PZ2020b} and \cite{PZ2016a,PZ2020a}, respectively.
Note that in the recent paper \cite{AH2020a}
similar results have been shown for no or full boundary conditions
using an alternative and more algebraic approach, 
the so-called Bernstein-Gelfand-Gelfand resolution (BGG). 

\section{Elasticity Complexes I}
\label{sec:ela1}%

Throughout this paper, let
$\om\subset\rt$ be a \emph{bounded strong Lipschitz domain}
with boundary $\ga$, decomposed into two parts $\gat$ and $\gan:=\ga\setminus\ol{\gat}$
with some \emph{relatively open and strong Lipschitz boundary part} $\gat\subset\ga$.

\subsection{Notations and Preliminaries}
\label{sec:not}%

We will strongly use the notations and results from 
our corresponding papers for the elasticity complex \cite{PZ2020b}
and for the de Rham complex \cite{PS2021a}.
In particular, we recall \cite[Section 2, Section 3]{PS2021a}
including the notion of \emph{extendable domains}.

We utilise the standard Sobolev spaces from \cite{PS2021a},
e.g., the usual Lebesgue and Sobolev spaces (scalar or tensor valued)
$\L{2}{}(\om)$ and $\H{k}{}(\om)$ with $k\in\nat_{0}$.
Boundary conditions are introduced in the \emph{strong sense} 
as closures of respective test fields, i.e.,
\begin{align*}
\H{k}{\gat}(\om)&:=\ol{\C{\infty}{\gat}(\om)}^{\H{k}{}(\om)},
\end{align*}
we well as in the \emph{weak sense} by
$$\bH{k}{\gat}(\om)
:=\big\{u\in\H{k}{}(\om):
\scp{\p^{\alpha}u}{\phi}_{\L{2}{}(\om)}=(-1)^{|\alpha|}\scp{u}{\p^{\alpha}\phi}_{\L{2}{}(\om)}
\quad\forall\,\phi\in\C{\infty}{\gan}(\om)\quad\forall\,|\alpha|\leq k\big\}.$$

\begin{lem}[{\cite[Lemma 3.2, Theorem 4.6]{PS2021a}}]
\label{lem:weakeqstrongderham}
$\bH{k}{\gat}(\om)=\H{k}{\gat}(\om)$, i.e.,
weak and strong boundary conditions coincide
for the standard Sobolev spaces with mixed boundary conditions.
\end{lem}

We shall use the abbreviations 
$\H{k}{\emptyset}(\om)=\H{k}{}(\om)$ and $\H{0}{\gat}(\om)=\L{2}{}(\om)$,
where the first notion is actually a density result
and incorporated into the notation by purpose.

\subsection{Operators}
\label{sec:elaop}%

Let
$\symGrad$, $\RotRott$, and $\Div$
(here $\Grad$, $\Rot$, and $\Div$ act row-wise as the operators 
$\grad$, $\rot$, and $\div$ from the vector de Rham complex)
be realised as densely defined (unbounded) linear operators
\begin{align*}
\rsymGradgat:D(\rsymGradgat)\subset\L{2}{}(\om)&\to\L{2}{\S}(\om);
&
v&\mapsto\sym\Grad v=\frac{1}{2}\big(\Grad v+(\Grad v)^{\top}\big),\\
\rRotRottSgat:D(\rRotRottSgat)\subset\L{2}{\S}(\om)&\to\L{2}{\S}(\om);
&
S&\mapsto\RotRott S=\Rot\big((\Rot S)^{\top}\big),\\
\rDivSgat:D(\rDivSgat)\subset\L{2}{\S}(\om)&\to\L{2}{}(\om);
&
T&\mapsto\Div T
\end{align*}
($S$, $T$, $\Grad v$, $\sym\Grad v$, $\Rot S$, $\RotRott S$ are $(3\times3)$-tensor fields, 
and $v$, $\Div T$ are $3$-vector fields) with domains of definition
$$D(\rsymGradgat):=\C{\infty}{\gat}(\om),\qquad
D(\rRotRottSgat):=\C{\infty}{\S,\gat}(\om),\qquad
D(\rDivSgat):=\C{\infty}{\S,\gat}(\om)$$
satisfying the complex properties
$$\rRotRottSgat\rsymGradgat\subset0,\qquad
\rDivSgat\rRotRottSgat\subset0.$$
For elementary properties of these operators see, e.g., \cite{PZ2020b},
in particular, we have the collection of formulas presented in Lemma \ref{lem:PZformulalem}.
Here, we introduce the Lebesgue Hilbert space
and the test space of symmetric tensor fields
$$\L{2}{\S}(\om)
:=\big\{S\in\L{2}{}(\om):S^{\top}=S\big\},\qquad
\C{\infty}{\S,\gat}(\om)
:=\C{\infty}{\gat}(\om)\cap\L{2}{\S}(\om),$$
respectively. We get the elasticity complex on smooth tensor fields
\begin{equation*}
\def\arrowlength{5ex}
\def\arrowdistance{0}
\begin{tikzcd}[column sep=\arrowlength]
\cdots 
\arrow[r, rightarrow, shift left=\arrowdistance, "\cdots"] 
& 
\L{2}{}(\om) 
\ar[r, rightarrow, shift left=\arrowdistance, "\rsymGradgat"] 
& 
[2.5em]
\L{2}{\S}(\om)
\arrow[r, rightarrow, shift left=\arrowdistance, "\rRotRottSgat"] 
& 
[2.5em]
\L{2}{\S}(\om)
\arrow[r, rightarrow, shift left=\arrowdistance, "\rDivSgat"] 
& 
[1em]
\L{2}{}(\om)
\arrow[r, rightarrow, shift left=\arrowdistance, "\cdots"] 
&
\cdots.
\end{tikzcd}
\end{equation*}

The closures 
$$\symGradgat:=\ol{\rsymGradgat},\qquad
\RotRottSgat:=\ol{\rRotRottSgat},\qquad
\DivSgat:=\ol{\rDivSgat}$$
and Hilbert space adjoints 
$$\symGradgats=\rsymGradgats,\qquad
(\RotRottSgat)^{*}=(\rRotRottSgat)^{*},\qquad
\DivSgats=\rDivSgats$$
are given by the densely defined and closed linear operators
\begin{align*}
\A_{0}:=\symGradgat:D(\symGradgat)\subset\L{2}{}(\om)&\to\L{2}{\S}(\om);
&
v&\mapsto\symGrad v,\\
\A_{1}:=\RotRottSgat:D(\RotRottSgat)\subset\L{2}{\S}(\om)&\to\L{2}{\S}(\om);
&
S&\mapsto\RotRott S,\\
\A_{2}:=\DivSgat:D(\DivSgat)\subset\L{2}{\S}(\om)&\to\L{2}{}(\om);
&
T&\mapsto\Div T,\\
\A_{0}^{*}=\symGradgats=-\bDivSgan:D(\bDivSgan)\subset\L{2}{\S}(\om)&\to\L{2}{}(\om);
&
S&\mapsto-\Div S,\\
\A_{1}^{*}=(\RotRottSgat)^{*}=\bRotRottSgan:D(\bRotRottSgan)\subset\L{2}{\S}(\om)&\to\L{2}{\S}(\om);
&
T&\mapsto\RotRott T,\\
\A_{2}^{*}=\DivSgats=-\bsymGradgan:D(\bsymGradgan)\subset\L{2}{}(\om)&\to\L{2}{\S}(\om);
&
v&\mapsto-\symGrad v
\end{align*}
with domains of definition
\begin{align*}
D(\A_{0})=D(\symGradgat)&=\H{}{\gat}(\symGrad,\om),
&
D(\A_{0}^{*})=D(\bDivSgan)&=\bH{}{\S,\gan}(\Div,\om),\\
D(\A_{1})=D(\RotRottSgat)&=\H{}{\S,\gat}(\RotRott\!\!,\om),
&
D(\A_{1}^{*})=D(\bRotRottSgan)&=\bH{}{\S,\gan}(\RotRott\!\!,\om),\\
D(\A_{2})=D(\DivSgat)&=\H{}{\S,\gat}(\Div,\om),
&
D(\A_{2}^{*})=D(\bsymGradgan)&=\bH{}{\gan}(\symGrad,\om).
\end{align*}
We shall introduce the latter Sobolev spaces in the next section.

\subsection{Sobolev Spaces}
\label{sec:sobolev}%

Let
\begin{align*}
\H{}{}(\symGrad,\om)
&:=\big\{v\in\L{2}{}(\om):\symGrad v\in\L{2}{}(\om)\big\},\\
\H{}{\S}(\RotRott\!\!,\om)
&:=\big\{S\in\L{2}{\S}(\om):\RotRott S\in\L{2}{}(\om)\big\},\\
\H{}{\S}(\Div,\om)
&:=\big\{T\in\L{2}{\S}(\om):\Div T\in\L{2}{}(\om)\big\}.
\end{align*}
Note that $M\in\H{}{\S}(\RotRott\!\!,\om)$ implies $\RotRott M\in\L{2}{\S}(\om)$,
and that we have by Korn's inequality the regularity 
$$\H{}{}(\symGrad,\om)=\H{1}{}(\om)$$
with equivalent norms.
Moreover, we define boundary conditions in the \emph{strong sense} 
as closures of respective test fields, i.e.,
\begin{align*}
\H{}{\gat}(\symGrad,\om)&:=\ol{\C{\infty}{\gat}(\om)}^{\H{}{}(\symGrad,\om)},\\
\H{}{\S,\gat}(\RotRott\!\!,\om)&:=\ol{\C{\infty}{\S,\gat}(\om)}^{\H{}{\S}(\RotRott\!\!,\om)},\\
\H{}{\S,\gat}(\Div,\om)&:=\ol{\C{\infty}{\S,\gat}(\om)}^{\H{}{\S}(\Div,\om)},
\end{align*}
and we have $\H{}{\emptyset}(\symGrad,\om)=\H{}{}(\symGrad,\om)=\H{1}{}(\om)$,
$\H{}{\S,\emptyset}(\RotRott\!\!,\om)=\H{}{\S}(\RotRott\!\!,\om)$,
and $\H{}{\S,\emptyset}(\Div,\om)=\H{}{\S}(\Div,\om)$,
which are density results and incorporated into the notation by purpose.
Spaces with vanishing $\RotRott$ and $\Div$
are denoted by $\H{}{\S,\gat,0}(\RotRott\!\!,\om)$ and $\H{}{\S,\gat,0}(\Div,\om)$, respectively. 
Note that, again by Korn's inequality, we have 
$$\H{}{\gat}(\symGrad,\om)=\H{1}{\gat}(\om).$$ 

We need also the Sobolev spaces with boundary conditions defined 
in the \emph{weak sense}, i.e.,
\begin{align*}
\bH{}{\gat}(\symGrad,\om)
&:=\big\{v\in\H{}{}(\symGrad,\om):
\scp{\symGrad v}{\Phi}_{\L{2}{}(\om)}=-\scp{v}{\Div\Phi}_{\L{2}{}(\om)}\\
&\hspace{60mm}\forall\,\Phi\in\C{\infty}{\S,\gan}(\om)\big\},\\
\bH{}{\S,\gat}(\RotRott\!\!,\om)
&:=\big\{S\in\H{}{\S}(\RotRott\!\!,\om):
\scp{\RotRott S}{\Psi}_{\L{2}{}(\om)}=\scp{S}{\RotRott\Psi}_{\L{2}{}(\om)}\\
&\hspace{60mm}\forall\,\Psi\in\C{\infty}{\S,\gan}(\om)\big\},\\
\bH{}{\S,\gat}(\Div,\om)
&:=\big\{T\in\H{}{\S}(\Div,\om):
\scp{\Div T}{\phi}_{\L{2}{}(\om)}=-\scp{T}{\symGrad\phi}_{\L{2}{}(\om)}\\
&\hspace{60mm}\forall\,\phi\in\C{\infty}{\gan}(\om)\big\}.
\end{align*}
Note that ``\emph{strong $\subset$ weak}'' holds, i.e., 
$\H{}{\cdots}(\cdots,\om)\subset\bH{}{\cdots}(\cdots,\om)$, e.g., 
$$\H{}{\S,\gat}(\RotRott\!\!,\om)\subset\bH{}{\S,\gat}(\RotRott\!\!,\om),\qquad
\H{}{\S,\gat}(\Div,\om)\subset\bH{}{\S,\gat}(\Div,\om),$$
and that the complex properties hold in both the strong and the weak case, e.g.,
$$\symGrad\H{}{\gat}(\om)\subset\H{}{\S,\gat,0}(\RotRott\!\!,\om),\qquad
\RotRott\bH{}{\S,\gat}(\RotRott\!\!,\om)\subset\bH{}{\S,\gat,0}(\Div,\om),$$
which follows immediately by the definitions.
In Remark \ref{rem:weakeqstrongela} below we comment on the question 
whether ``\emph{strong $=$ weak}'' holds in general.

\subsection{Higher Order Sobolev Spaces}
\label{sec:highsobolev}%

For $k\in\nat_{0}$ we define higher order Sobolev spaces by
\begin{align*}
\H{k}{\S}(\om)
&:=\H{k}{}(\om)\cap\L{2}{\S}(\om),\\
\H{k}{\S,\gat}(\om)
&:=\ol{\C{\infty}{\S,\gat}(\om)}^{\H{k}{}(\om)}
=\H{k}{\gat}(\om)\cap\L{2}{\S}(\om),\\
\H{k}{}(\symGrad,\om)
&:=\big\{v\in\H{k}{}(\om):\symGrad v\in\H{k}{}(\om)\big\},\\
\H{k}{\gat}(\symGrad,\om)
&:=\big\{v\in\H{k}{\gat}(\om)\cap\H{}{\gat}(\symGrad,\om):\symGrad v\in\H{k}{\gat}(\om)\big\},\\
\H{k}{\S}(\RotRott\!\!,\om)
&:=\big\{S\in\H{k}{\S}(\om):\RotRott S\in\H{k}{}(\om)\big\},\\
\H{k}{\S,\gat}(\RotRott\!\!,\om)
&:=\big\{S\in\H{k}{\gat}(\om)\cap\H{}{\S,\gat}(\RotRott\!\!,\om):\RotRott S\in\H{k}{\gat}(\om)\big\},\\
\H{k}{\S}(\Div,\om)
&:=\big\{T\in\H{k}{\S}(\om):\Div T\in\H{k}{}(\om)\big\},\\
\H{k}{\S,\gat}(\Div,\om)
&:=\big\{T\in\H{k}{\gat}(\om)\cap\H{}{\S,\gat}(\Div,\om):\Div T\in\H{k}{\gat}(\om)\big\}.
\end{align*}
We see $\H{k}{\S,\emptyset}(\RotRott\!\!,\om)=\H{k}{\S}(\RotRott\!\!,\om)$
and $\H{0}{\S,\emptyset}(\RotRott\!\!,\om)=\H{}{\S}(\RotRott\!\!,\om)$
as well as $\H{0}{\S,\gat}(\RotRott\!\!,\om)=\H{}{\S,\gat}(\RotRott\!\!,\om)$.
Note that for $\gat\neq\emptyset$ it holds
\begin{align}
\label{defsobolevkgeqtwo}
\H{k}{\S,\gat}(\RotRott\!\!,\om)
&=\big\{S\in\H{k}{\S,\gat}(\om):\RotRott S\in\H{k}{\gat}(\om)\big\},\qquad
k\geq2,
\end{align}
but for $\gat\neq\emptyset$ and $k=0$ and $k=1$ \big(as $\H{0}{\S,\gat}(\om)=\L{2}{\S}(\om)$\big)
\begin{align*}
\H{0}{\S,\gat}(\RotRott\!\!,\om)
&=\H{}{\S,\gat}(\RotRott\!\!,\om)\\
&\subsetneq\big\{S\in\H{0}{\S,\gat}(\om):\RotRott S\in\H{0}{\gat}(\om)\big\}
=\H{}{\S}(\RotRott\!\!,\om),\\
\H{1}{\S,\gat}(\RotRott\!\!,\om)
&\subsetneq\big\{S\in\H{1}{\S,\gat}(\om):\RotRott S\in\H{1}{\gat}(\om)\big\},
\end{align*}
respectively. As before, we introduce the kernels
\begin{align*}
\H{k}{\S,\gat,0}(\RotRott\!\!,\om)
&:=\H{k}{\gat}(\om)\cap\H{}{\S,\gat,0}(\RotRott\!\!,\om)
=\H{k}{\S,\gat}(\RotRott\!\!,\om)\cap\H{}{\S,0}(\RotRott\!\!,\om)\\
&\;=\big\{S\in\H{k}{\S,\gat}(\RotRott\!\!,\om):\RotRott S=0\big\}.
\end{align*}
The corresponding remarks and definitions extend 
to the $\H{k}{\gat}(\symGrad,\om)$-spaces and $\H{k}{\S,\gat}(\Div,\om)$-spaces as well.
In particular, we have for $\gat\neq\emptyset$ and $k\geq1$
\begin{align}
\begin{aligned}
\label{defsobolevkgeqone}
\H{k}{\gat}(\symGrad,\om)
&=\big\{v\in\H{k}{\gat}(\om):\symGrad v\in\H{k}{\gat}(\om)\big\},\\
\H{k}{\S,\gat}(\Div,\om)
&=\big\{T\in\H{k}{\S,\gat}(\om):\Div T\in\H{k}{\gat}(\om)\big\},
\end{aligned}
\end{align}
and 
\begin{align*}
\H{0}{\gat}(\symGrad,\om)
=\H{}{\gat}(\symGrad,\om)
&\subsetneq\big\{v\in\H{0}{\gat}(\om):\symGrad v\in\H{0}{\gat}(\om)\big\}
=\H{}{}(\symGrad,\om),\\
\H{0}{\S,\gat}(\Div,\om)
=\H{}{\S,\gat}(\Div,\om)
&\subsetneq\big\{T\in\H{0}{\S,\gat}(\om):\Div T\in\H{0}{\gat}(\om)\big\}
=\H{}{\S}(\Div,\om),
\end{align*}
as well as
\begin{align*}
\H{k}{\S,\gat,0}(\Div,\om)
&=\H{k}{\gat}(\om)\cap\H{}{\S,\gat,0}(\Div,\om)
=\H{k}{\S,\gat}(\Div,\om)\cap\H{}{\S,0}(\Div,\om)\\
&=\big\{T\in\H{k}{\S,\gat}(\Div,\om):\Div T=0\big\}.
\end{align*}

Analogously, we define the Sobolev spaces 
$\bH{k}{\gat}(\symGrad,\om)$, $\bH{k}{\S,\gat}(\RotRott\!\!,\om)$, $\bH{k}{\S,\gat}(\Div,\om)$,
and $\bH{k}{\S,\gat,0}(\RotRott\!\!,\om)$, $\bH{k}{\S,\gat,0}(\Div,\om)$
using the respective Sobolev spaces with weak boundary conditions.
Note that again ``\emph{strong $\subset$ weak}'' holds, i.e., 
$\H{\cdots}{\cdots}(\cdots,\om)\subset\bH{\cdots}{\cdots}(\cdots,\om)$, e.g., 
$$\H{k}{\S,\gat}(\RotRott\!\!,\om)\subset\bH{k}{\S,\gat}(\RotRott\!\!,\om),\qquad
\H{k}{\S,\gat}(\Div,\om)\subset\bH{k}{\S,\gat}(\Div,\om),$$
and that the complex properties hold in both the strong and the weak case, e.g.,
$$\symGrad\H{k+1}{\gat}(\om)\subset\H{k}{\S,\gat,0}(\RotRott\!\!,\om),\qquad
\RotRott\bH{k}{\S,\gat}(\RotRott\!\!,\om)\subset\bH{k}{\S,\gat,0}(\Div,\om).$$
Moreover, the corresponding results for \eqref{defsobolevkgeqtwo} and \eqref{defsobolevkgeqone}
hold for the weak spaces as well.

In the forthcoming sections we shall also investigate
whether indeed ``\emph{strong $=$ weak}'' holds.
We start with a simple implication from Lemma \ref{lem:weakeqstrongderham}.

\begin{cor}
\label{cor:weakeqstrongderham}
$\bH{k}{\S,\gat}(\om)=\H{k}{\S,\gat}(\om)$, i.e.,
weak and strong boundary conditions coincide
for the standard Sobolev spaces of symmetric tensor fields 
with mixed boundary conditions.
\end{cor}

Lemma \ref{lem:weakeqstrongderham}, Corollary \ref{cor:weakeqstrongderham},
\eqref{defsobolevkgeqtwo}, \eqref{defsobolevkgeqone}, and Korn's inequality
show the following.

\begin{lem}[higher order weak and strong partial boundary conditions coincide]
\label{lem:weakeqstrongela}
\mbox{}
\begin{itemize}
\item[\bf(i)]
For $k\geq0$ it holds
$\H{k}{\gat}(\symGrad,\om)=\H{k+1}{\gat}(\om)=\bH{k+1}{\gat}(\om)$.
\item[\bf(ii)]
For $k\geq1$ it holds
\begin{align*}
\bH{k}{\gat}(\symGrad,\om)
&=\big\{v\in\H{k}{\gat}(\om):\symGrad v\in\H{k}{\gat}(\om)\big\}
=\H{k}{\gat}(\symGrad,\om)
=\H{k+1}{\gat}(\om),\\
\bH{k}{\S,\gat}(\Div,\om)
&=\big\{T\in\H{k}{\S,\gat}(\om):\Div T\in\H{k}{\gat}(\om)\big\}
=\H{k}{\S,\gat}(\Div,\om).
\end{align*}
\item[\bf(iii)]
For $k\geq2$ it holds
\begin{align*}
\bH{k}{\S,\gat}(\RotRott\!\!,\om)
&=\big\{S\in\H{k}{\S,\gat}(\om):\RotRott S\in\H{k}{\gat}(\om)\big\}
=\H{k}{\S,\gat}(\RotRott\!\!,\om).
\end{align*}
\end{itemize}
\end{lem}

\begin{rem}[weak and strong partial boundary conditions coincide]
\label{rem:weakeqstrongela}
In \cite{PZ2020b} we could prove the corresponding results ``\emph{strong $=$ weak}'' for the whole elasticity complex
but only with empty or full boundary conditions ($\ga_{t}=\emptyset$ or $\ga_{t}=\ga$).
Therefore, in these special cases, the adjoints are well-defined on the spaces with strong boundary conditions as well.

Lemma \ref{lem:weakeqstrongela} shows that
for higher values of $k$ indeed ``\emph{strong $=$ weak}'' holds.
Thus to show ``\emph{strong $=$ weak}'' in general 
we only have to prove that equality holds in the remains cases
$k=0$ and $k=1$, i.e.,
we only have to show
\begin{align*}
\bH{}{\gat}(\symGrad,\om)
&\subset\H{}{\gat}(\symGrad,\om),
&
\bH{}{\S,\gat}(\RotRott\!\!,\om)
&\subset\H{}{\S,\gat}(\RotRott\!\!,\om),\\
\bH{}{\S,\gat}(\Div,\om)
&\subset\H{k}{\S,\gat}(\Div,\om),
&
\bH{1}{\S,\gat}(\RotRott\!\!,\om)
&\subset\H{1}{\S,\gat}(\RotRott\!\!,\om).
\end{align*}

The most delicate situation appears due to the second order nature of $\RotRottS$.
In Corollary \ref{cor:weakstrongela} we shall show 
using regular decompositions that these results 
(weak and strong boundary conditions coincide for the elasticity complex for all $k\geq0$)
indeed hold true.
\end{rem}

\subsection{More Sobolev Spaces}
\label{sec:sobolevmore}%

For $k\in\nat$ we introduce also slightly less regular higher order Sobolev spaces by
\begin{align*}
\H{k,k-1}{\S,\gat}(\RotRott\!\!,\om)
&:=\big\{S\in\H{k}{\gat}(\om)\cap\H{}{\S,\gat}(\RotRott\!\!,\om):\RotRott S\in\H{k-1}{\gat}(\om)\big\},\\
\bH{k,k-1}{\S,\gat}(\RotRott\!\!,\om)
&:=\big\{S\in\bH{k}{\gat}(\om)\cap\bH{}{\S,\gat}(\RotRott\!\!,\om):\RotRott S\in\bH{k-1}{\gat}(\om)\big\},
\end{align*}
and we extend all conventions of our notations.
For the kernels we have
\begin{align*}
\H{k,k-1}{\S,\gat,0}(\RotRott\!\!,\om)
&=\H{k}{\S,\gat,0}(\RotRott\!\!,\om),
&
\bH{k,k-1}{\S,\gat,0}(\RotRott\!\!,\om)
&=\bH{k}{\S,\gat,0}(\RotRott\!\!,\om).
\end{align*}
Note that, as before, 
the intersection with $\H{}{\S,\gat}(\RotRott\!\!,\om)$ and $\bH{}{\S,\gat}(\RotRott\!\!,\om)$
is only needed if $k=1$.
Again we have ``\emph{strong $\subset$ weak}'', i.e., 
$\H{k,k-1}{\S,\gat}(\RotRott\!\!,\om)\subset\bH{k,k-1}{\S,\gat}(\RotRott\!\!,\om)$,
and in both cases (weak and strong) the complex properties hold, e.g.,
$$\symGrad\H{k+1}{\gat}(\om)\subset\H{k,k-1}{\S,\gat,0}(\RotRott\!\!,\om),\qquad
\RotRott\bH{k,k-1}{\S,\gat}(\RotRott\!\!,\om)\subset\bH{k-1}{\S,\gat,0}(\Div,\om).$$

Similar to Lemma \ref{lem:weakeqstrongela} we have the following.

\begin{lem}[higher order weak and strong partial boundary conditions coincide]
\label{lem:weakeqstrongela2}
For $k\geq2$
\begin{align*}
\bH{k,k-1}{\S,\gat}(\RotRott\!\!,\om)
&=\big\{S\in\H{k}{\S,\gat}(\om):\RotRott S\in\H{k-1}{\gat}(\om)\big\}
=\H{k,k-1}{\S,\gat}(\RotRott\!\!,\om).
\end{align*}
\end{lem}

\subsection{Some Elasticity Complexes}
\label{sec:elacomplexes}%

By definition we have densely defined and closed (unbounded) linear operators 
defining three dual pairs 
\begin{align*}
(\symGradgat,\symGradgats)
&=(\symGradgat,-\bDivSgan),\\
\big(\RotRottSgat,(\RotRottSgat)^{*}\big)
&=(\RotRottSgat,\bRotRottSgan),\\
(\DivSgat,\DivSgats)
&=(\DivSgat,-\bsymGradgan).
\end{align*}
\cite[Remark 2.5, Remark 2.6]{PS2021a} show the complex properties 
\begin{align*}
\RotRottSgat\symGradgat&\subset0,
&
\DivSgat\RotRottSgat&\subset0,\\
-\bDivSgan\bRotRottSgan&\subset0,
&
-\bRotRottSgan\bsymGradgan&\subset0.
\end{align*}
Hence we get the primal and dual elasticity Hilbert complex
\begin{equation}
\label{elacomplex1}
\def\arrowlength{16ex}
\def\arrowdistance{.8}
\begin{tikzcd}[column sep=\arrowlength]
\cdots
\arrow[r, rightarrow, shift left=\arrowdistance, "\cdots"] 
\arrow[r, leftarrow, shift right=\arrowdistance, "\cdots"']
& 
[-5em]
\L{2}{}(\om) 
\ar[r, rightarrow, shift left=\arrowdistance, "\symGradgat"] 
\ar[r, leftarrow, shift right=\arrowdistance, "-\bDivSgan"']
&
[-1em]
\L{2}{\S}(\om) 
\ar[r, rightarrow, shift left=\arrowdistance, "\RotRottSgat"] 
\ar[r, leftarrow, shift right=\arrowdistance, "\bRotRottSgan"']
& 
[-1em]
\L{2}{\S}(\om) 
\arrow[r, rightarrow, shift left=\arrowdistance, "\DivSgat"] 
\arrow[r, leftarrow, shift right=\arrowdistance, "-\bsymGradgan"']
& 
[-1em]
\L{2}{}(\om) 
\arrow[r, rightarrow, shift left=\arrowdistance, "\cdots"] 
\arrow[r, leftarrow, shift right=\arrowdistance, "\cdots"']
&
[-5em]
\cdots
\end{tikzcd}
\end{equation}
with the complex properties 
\begin{align*}
R(\symGradgat)&\subset N(\RotRottSgat),
&
R(\bRotRottSgan)&\subset N(\bDivSgan),\\
R(\RotRottSgat)&\subset N(\DivSgat),
&
R(\bsymGradgan)&\subset N(\bRotRottSgan).
\end{align*}
The long primal and dual elasticity Hilbert complex,
cf.~\cite[(12)]{PS2021a}, reads
\begin{equation}
\label{elacomplex2}
\footnotesize
\def\arrowlength{16ex}
\def\arrowdistance{.8}
\begin{tikzcd}[column sep=\arrowlength]
\RM_{\gat}
\arrow[r, rightarrow, shift left=\arrowdistance, "\iota_{\RM_{\gat}}"] 
\arrow[r, leftarrow, shift right=\arrowdistance, "\pi_{\RM_{\gat}}"']
& 
[-2em]
\L{2}{}(\om) 
\ar[r, rightarrow, shift left=\arrowdistance, "\symGradgat"] 
\ar[r, leftarrow, shift right=\arrowdistance, "-\bDivSgan"']
&
[-1em]
\L{2}{\S}(\om) 
\ar[r, rightarrow, shift left=\arrowdistance, "\RotRottSgat"] 
\ar[r, leftarrow, shift right=\arrowdistance, "\bRotRottSgan"']
& 
[0em]
\L{2}{\S}(\om) 
\arrow[r, rightarrow, shift left=\arrowdistance, "\DivSgat"] 
\arrow[r, leftarrow, shift right=\arrowdistance, "-\bsymGradgan"']
& 
[-1em]
\L{2}{}(\om) 
\arrow[r, rightarrow, shift left=\arrowdistance, "\pi_{\RM_{\gan}}"] 
\arrow[r, leftarrow, shift right=\arrowdistance, "\iota_{\RM_{\gan}}"']
&
[-2em]
\RM_{\gan}
\end{tikzcd}
\end{equation}
with the additional complex properties 
\begin{align*}
R(\iota_{\RM_{\gat}
})&=N(\symGradgat)=\RM_{\gat},
&
\ol{R(\bDivSgan)}&=\RM_{\gat}^{\bot_{\L{2}{}(\om)}},\\
\ol{R(\DivSgat)}&=\RM_{\gan}^{\bot_{\L{2}{}(\om)}},
&
R(\iota_{\RM_{\gan}})&=N(\bsymGradgan)=\RM_{\gan},
\end{align*}
where
$$\RM_{\Sigma}
=\begin{cases}
\{0\}&\text{if }\Sigma\neq\emptyset,\\
\RM&\text{if }\Sigma=\emptyset,
\end{cases}\qquad\text{with}\qquad
\RM:=\big\{x\mapsto Qx+q:Q\in\reals^{3\times3}\,\text{skew},\,q\in\reals^{3}\big\}$$
denoting the global rigid motions in $\om$. Note that $\dim\RM=6$.

More generally, in addition to \eqref{elacomplex2},
we shall discuss for $k\in\nat_{0}$ the higher Sobolev order 
(long primal and formally dual) elasticity Hilbert complexes
(omitting $\om$ in the notation)

\begin{equation*}
\footnotesize
\def\arrowlength{16ex}
\def\arrowdistance{0}
\begin{tikzcd}[column sep=\arrowlength]
\RM_{\gat}
\arrow[r, rightarrow, shift left=\arrowdistance, "\iota_{\RM_{\gat}}"] 
& 
[-3em]
\H{k}{\gat}
\ar[r, rightarrow, shift left=\arrowdistance, "\symGradgatk"] 
&
[-1em]
\H{k}{\S,\gat}
\ar[r, rightarrow, shift left=\arrowdistance, "\RotRottSgatk"] 
& 
[-1em]
\H{k}{\S,\gat}
\arrow[r, rightarrow, shift left=\arrowdistance, "\DivSgatk"] 
& 
[-3em]
\H{k}{\gat}
\arrow[r, rightarrow, shift left=\arrowdistance, "\pi_{\RM_{\gan}}"] 
&
[-3em]
\RM_{\gan},
\end{tikzcd}
\end{equation*}
\vspace*{-4mm}
\begin{equation*}
\footnotesize
\def\arrowlength{16ex}
\def\arrowdistance{0}
\begin{tikzcd}[column sep=\arrowlength]
\RM_{\gat}
\arrow[r, leftarrow, shift right=\arrowdistance, "\pi_{\RM_{\gat}}"]
& 
[-3em]
\H{k}{\gan}
\ar[r, leftarrow, shift right=\arrowdistance, "-\bDivSgank"]
&
[-2em]
\H{k}{\S,\gan}
\ar[r, leftarrow, shift right=\arrowdistance, "\bRotRottSgank"]
& 
[0em]
\H{k}{\S,\gan}
\arrow[r, leftarrow, shift right=\arrowdistance, "-\bsymGradgank"]
& 
[0em]
\H{k}{\gan}
\arrow[r, leftarrow, shift right=\arrowdistance, "\iota_{\RM_{\gan}}"]
&
[-3em]
\RM_{\gan}
\end{tikzcd}
\end{equation*}
with associated domain complexes
\begin{equation*}
\scriptsize
\def\arrowlength{16ex}
\def\arrowdistance{0}
\begin{tikzcd}[column sep=\arrowlength]
\RM_{\gat}
\arrow[r, rightarrow, shift left=\arrowdistance, "\iota_{\RM_{\gat}}"] 
& 
[-3em]
\H{k}{\gat}(\symGrad)
\ar[r, rightarrow, shift left=\arrowdistance, "\symGradgatk"] 
&
[-1em]
\H{k}{\S,\gat}(\RotRott)
\ar[r, rightarrow, shift left=\arrowdistance, "\RotRottSgatk"] 
& 
[-1em]
\H{k}{\S,\gat}(\Div)
\arrow[r, rightarrow, shift left=\arrowdistance, "\DivSgatk"] 
& 
[-3em]
\H{k}{\gat}
\arrow[r, rightarrow, shift left=\arrowdistance, "\pi_{\RM_{\gan}}"] 
&
[-3em]
\RM_{\gan},
\end{tikzcd}
\end{equation*}
\vspace*{-5mm}
\begin{equation*}
\scriptsize
\def\arrowlength{16ex}
\def\arrowdistance{0}
\begin{tikzcd}[column sep=\arrowlength]
\RM_{\gat}
\arrow[r, leftarrow, shift right=\arrowdistance, "\pi_{\RM_{\gat}}"]
& 
[-3em]
\H{k}{\gat}
\ar[r, leftarrow, shift right=\arrowdistance, "-\bDivSgank"]
&
[-2em]
\bH{k}{\S,\gan}(\Div)
\ar[r, leftarrow, shift right=\arrowdistance, "\bRotRottSgank"]
& 
[-1em]
\bH{k}{\S,\gan}(\RotRott)
\arrow[r, leftarrow, shift right=\arrowdistance, "-\bsymGradgank"]
& 
[0em]
\bH{k}{\gan}(\symGrad)
\arrow[r, leftarrow, shift right=\arrowdistance, "\iota_{\RM_{\gan}}"]
&
[-3em]
\RM_{\gan}.
\end{tikzcd}
\end{equation*}

Additionally, for $k\geq1$ we will also discuss the following variants of the elasticity complexes 
\begin{equation*}
\footnotesize
\def\arrowlength{16ex}
\def\arrowdistance{0}
\begin{tikzcd}[column sep=\arrowlength]
\RM_{\gat}
\arrow[r, rightarrow, shift left=\arrowdistance, "\iota_{\RM_{\gat}}"] 
& 
[-3em]
\H{k}{\gat}
\ar[r, rightarrow, shift left=\arrowdistance, "\symGradgatk"] 
&
[-1em]
\H{k}{\S,\gat}
\ar[r, rightarrow, shift left=\arrowdistance, "\RotRottSgatkkmo"] 
& 
[0em]
\H{k-1}{\S,\gat}
\arrow[r, rightarrow, shift left=\arrowdistance, "\DivSgatkmo"] 
& 
[-3em]
\H{k-1}{\gat}
\arrow[r, rightarrow, shift left=\arrowdistance, "\pi_{\RM_{\gan}}"] 
&
[-3em]
\RM_{\gan},
\end{tikzcd}
\end{equation*}
\vspace*{-4mm}
\begin{equation*}
\footnotesize
\def\arrowlength{16ex}
\def\arrowdistance{0}
\begin{tikzcd}[column sep=\arrowlength]
\RM_{\gat}
\arrow[r, leftarrow, shift right=\arrowdistance, "\pi_{\RM_{\gat}}"]
& 
[-3em]
\H{k-1}{\gan}
\ar[r, leftarrow, shift right=\arrowdistance, "-\bDivSgankmo"]
&
[-2em]
\H{k-1}{\S,\gan}
\ar[r, leftarrow, shift right=\arrowdistance, "\bRotRottSgankkmo"]
& 
[0em]
\H{k}{\S,\gan}
\arrow[r, leftarrow, shift right=\arrowdistance, "-\bsymGradgank"]
& 
[0em]
\H{k}{\gan}
\arrow[r, leftarrow, shift right=\arrowdistance, "\iota_{\RM_{\gan}}"]
&
[-3em]
\RM_{\gan}
\end{tikzcd}
\end{equation*}
with associated domain complexes
\begin{equation*}
\scriptsize
\def\arrowlength{16ex}
\def\arrowdistance{0}
\begin{tikzcd}[column sep=\arrowlength]
\RM_{\gat}
\arrow[r, rightarrow, shift left=\arrowdistance, "\iota_{\RM_{\gat}}"] 
& 
[-3em]
\H{k}{\gat}(\symGrad)
\ar[r, rightarrow, shift left=\arrowdistance, "\symGradgatk"] 
&
[-1em]
\H{k,k-1}{\S,\gat}(\RotRott)
\ar[r, rightarrow, shift left=\arrowdistance, "\RotRottSgatkkmo"] 
& 
[0em]
\H{k-1}{\S,\gat}(\Div)
\arrow[r, rightarrow, shift left=\arrowdistance, "\DivSgatkmo"] 
& 
[-3em]
\H{k-1}{\gat}
\arrow[r, rightarrow, shift left=\arrowdistance, "\pi_{\RM_{\gan}}"] 
&
[-3em]
\RM_{\gan},
\end{tikzcd}
\end{equation*}
\vspace*{-5mm}
\begin{equation*}
\scriptsize
\def\arrowlength{16ex}
\def\arrowdistance{0}
\begin{tikzcd}[column sep=\arrowlength]
\RM_{\gat}
\arrow[r, leftarrow, shift right=\arrowdistance, "\pi_{\RM_{\gat}}"]
& 
[-3em]
\H{k-1}{\gat}
\ar[r, leftarrow, shift right=\arrowdistance, "-\bDivSgankmo"]
&
[-2em]
\bH{k-1}{\S,\gan}(\Div)
\ar[r, leftarrow, shift right=\arrowdistance, "\bRotRottSgankkmo"]
& 
[0em]
\bH{k,k-1}{\S,\gan}(\RotRott)
\arrow[r, leftarrow, shift right=\arrowdistance, "-\bsymGradgank"]
& 
[0em]
\bH{k}{\gan}(\symGrad)
\arrow[r, leftarrow, shift right=\arrowdistance, "\iota_{\RM_{\gan}}"]
&
[-3em]
\RM_{\gan}.
\end{tikzcd}
\end{equation*}
Here we have introduced the densely defined and closed linear operators
\begin{align*}
\symGradgatk:D(\symGradgatk)\subset\H{k}{\gat}(\om)&\to\H{k}{\S,\gat}(\om);
&
v&\mapsto\symGrad v,\\
\RotRottSgatk:D(\RotRottSgatk)\subset\H{k}{\S,\gat}(\om)&\to\H{k}{\S,\gat}(\om);
&
S&\mapsto\RotRott S,\\
\DivSgatk:D(\DivSgatk)\subset\H{k}{\S,\gat}(\om)&\to\H{k}{\gat}(\om);
&
T&\mapsto\Div T,\\
-\bDivSgank:D(\bDivSgank)\subset\H{k}{\S,\gan}(\om)&\to\H{k}{\gan}(\om);
&
S&\mapsto-\Div S,\\
\bRotRottSgank:D(\bRotRottSgank)\subset\H{k}{\S,\gan}(\om)&\to\H{k}{\S,\gan}(\om);
&
T&\mapsto\RotRott T,\\
-\bsymGradgank:D(\bsymGradgank)\subset\H{k}{\gan}(\om)&\to\H{k}{\S,\gan}(\om);
&
v&\mapsto-\symGrad v
\end{align*}
with domains of definition
\begin{align*}
D(\symGradgatk)&=\H{k}{\gat}(\symGrad,\om),
&
D(\bDivSgank)&=\bH{k}{\S,\gan}(\Div,\om),\\
D(\RotRottSgatk)&=\H{k}{\S,\gat}(\RotRott\!\!,\om),
&
D(\bRotRottSgank)&=\bH{k}{\S,\gan}(\RotRott\!\!,\om),\\
D(\DivSgatk)&=\H{k}{\S,\gat}(\Div,\om),
&
D(\bsymGradgank)&=\bH{k}{\gan}(\symGrad,\om),
\end{align*}
as well as 
\begin{align*}
\RotRottSgatkkmo:D(\RotRottSgatkkmo)\subset\H{k}{\S,\gat}(\om)&\to\H{k-1}{\S,\gat}(\om);
&
S&\mapsto\RotRott S,\\
\bRotRottSgankkmo:D(\bRotRottSgankkmo)\subset\H{k}{\S,\gan}(\om)&\to\H{k-1}{\S,\gan}(\om);
&
T&\mapsto\RotRott T
\end{align*}
with domains of definition
\begin{align*}
D(\RotRottSgatkkmo)&=\H{k,k-1}{\S,\gat}(\RotRott\!\!,\om),
&
D(\bRotRottSgankkmo)&=\bH{k,k-1}{\S,\gan}(\RotRott\!\!,\om).
\end{align*}

\subsection{Dirichlet/Neumann Fields}
\label{sec:dirneu}%

We also introduce the cohomology space of 
elastic Dirichlet/Neumann tensor fields (generalised harmonic tensors)
$$\Harm{}{\S,\gat,\gan,\eps}(\om)
:=N(\RotRottSgat)\cap N(\DivSgan\eps)
=\H{}{\S,\gat,0}(\RotRott,\om)\cap\eps^{-1}\H{}{\S,\gan,0}(\Div,\om).$$
Here, $\eps:\L{2}{\S}(\om)\to\L{2}{\S}(\om)$ is a symmetric and positive 
topological isomorphism (symmetric and positive bijective bounded linear operator),
which introduces a new inner product
$$\scp{\,\cdot\,}{\,\cdot\,}_{\L{2}{\S,\eps}(\om)}
:=\scp{\eps\,\cdot\,}{\,\cdot\,}_{\L{2}{\S}(\om)},$$
where $\L{2}{\S,\eps}(\om):=\L{2}{\S}(\om)$ (as linear space)
equipped with the inner product $\scp{\,\cdot\,}{\,\cdot\,}_{\L{2}{\S,\eps}(\om)}$.
Such \emph{weights} $\eps$ shall be called \emph{admissible}
and a typical example is given by a 
symmetric, $\L{\infty}{}$-bounded, and uniformly positive definite tensor field
$\eps:\om\to\reals^{(3\times3)\times(3\times3)}$.

\section{Elasticity Complexes II}
\label{sec:ela2}%

\subsection{Regular Potentials and Decompositions I}
\label{sec:regpotdeco1}%

\subsubsection{Extendable Domains}
\label{sec:regpotdecoextdom}%

\begin{theo}[regular potential operators for extendable domains]
\label{highorderregpotextdomela}
Let $(\om,\gat)$ be an extendable bounded strong Lipschitz pair
and let $k\geq0$. Then there exist bounded linear regular potential operators
\begin{align*}
\PotP_{\symGrad,\gat}^{k}:
\bH{k}{\S,\gat,0}(\RotRott,\om)
&\To\H{k+1}{\gat}(\om)\cap\H{k+1}{}(\rt),\\
\PotP_{\RotRottS,\gat}^{k}:
\bH{k}{\S,\gat,0}(\Div,\om)
&\To\H{k+2}{\S,\gat}(\om)\cap\H{k+2}{}(\rt),\\
\PotP_{\DivS,\gat}^{k}:
\H{k}{\gat}(\om)\cap(\RM_{\gan})^{\bot_{\L{2}{}(\om)}}
&\To\H{k+1}{\S,\gat}(\om)\cap\H{k+1}{}(\rt).
\end{align*}
In particular, $\PotP_{\dots}^{\dots}$ are right inverses for $\symGrad$, $\RotRott$, and $\Div$, respectively, i.e.,
\begin{align*}
\symGrad\PotP_{\symGrad,\gat}^{k}
&=\id_{\bH{k}{\S,\gat,0}(\RotRott,\om)},\\
\RotRott\PotP_{\RotRottS,\gat}^{k}
&=\id_{\bH{k}{\S,\gat,0}(\Div,\om)},\\
\Div\PotP_{\DivS,\gat}^{k}
&=\id_{\H{k}{\gat}(\om)\cap(\RM_{\gan})^{\bot_{\L{2}{}(\om)}}}.
\end{align*}
Without loss of generality, $\PotP_{\dots}^{\dots}$ 
map to tensor fields with a fixed compact support in $\rt$.
\end{theo}

\begin{rem}
Note that $\A_{n}\PotP_{\A_{n}}=\id_{R(\A_{n})}$ is a general property 
of a (bounded regular) potential operator $\PotP_{\A_{n}}:R(\A_{n})\to\H{+}{n}$
with $\H{+}{n}\subset D(\A_{n})$.
\end{rem}

\begin{proof}[Proof of Theorem \ref{highorderregpotextdomela}]
In \cite[Theorem 4.2]{PZ2020b} we have shown the stated results for $\gat=\ga$ and $\gat=\emptyset$,
which is also a crucial ingredient of this proof.
Note that in these two special cases always ``\emph{strong $=$ weak}'' holds
as $\A_{n}^{**}=\ol{\A_{n}}=\A_{n}$, and that this argument fails in the remaining cases
of mixed boundary conditions. Therefore, let $\emptyset\varsubsetneq\gat\varsubsetneq\ga$.
Moreover, recall the notion of an extendable domain from \cite[Section 3]{PS2021a}.
In particular, $\widehat{\om}$ and the extended domain 
$\widetilde{\om}$ are topologically trivial.
\begin{itemize}
\item
Let $S\in\bH{k}{\S,\gat,0}(\RotRott,\om)$. 
By definition, $S$ can be extended through $\gat$
by zero to the larger domain $\widetilde{\om}$ yielding
$$\widetilde{S}\in\bH{k}{\S,\emptyset,0}(\RotRott,\widetilde{\om})
=\bH{k}{\S,0}(\RotRott,\widetilde{\om})
=\H{k}{\S,0}(\RotRott,\widetilde{\om}).$$
By \cite[Theorem 4.2]{PZ2020b} there exists $\widetilde{v}\in\H{k+1}{}(\rt)$
such that $\symGrad\widetilde{v}=\widetilde{S}$ in $\widetilde{\om}$.
Since $\widetilde{S}=0$ in $\widehat{\om}$, $\widetilde{v}$ must be a rigid motion 
$r\in\RM$ in $\widehat{\om}$. Far outside of $\widetilde{\om}$ we modify $r$
by a cut-off function such that the resulting vector field $\widetilde{r}$ 
is compactly supported and $\widetilde{r}|_{\widetilde{\om}}=r$.
Then $v:=\widetilde{v}-\widetilde{r}\in\H{k+1}{}(\rt)$
with $v|_{\widehat{\om}}=0$. Hence $v$ belongs to $\H{k+1}{\gat}(\om)$
and depends continuously on $S$. Moreover, $v$ satisfies
$\symGrad v=\symGrad\widetilde{v}=\widetilde{S}$ in $\widetilde{\om}$,
in particular $\symGrad v=S$ in $\om$.
We put $\PotP_{\symGrad,\gat}^{k}S:=v\in\H{k+1}{\gat}(\om)$.
\item
Let $T\in\bH{k}{\S,\gat,0}(\Div,\om)$. By definition, $T$ can be extended through $\gat$
by zero to $\widetilde{\om}$ giving
$$\widetilde{T}\in\bH{k}{\S,\emptyset,0}(\Div,\widetilde{\om})
=\bH{k}{\S,0}(\Div,\widetilde{\om})
=\H{k}{\S,0}(\Div,\widetilde{\om}).$$
By \cite[Theorem 4.2]{PZ2020b} there exists $\widetilde{S}\in\H{k+2}{\S}(\rt)$
such that $\RotRott\widetilde{S}=\widetilde{T}$ in $\widetilde{\om}$.
Since $\widetilde{T}=0$ in $\widehat{\om}$, i.e.,
$\widetilde{S}|_{\widehat{\om}}\in\H{k+2}{\S,0}(\RotRott,\widehat{\om})$,
we get again by \cite[Theorem 4.2]{PZ2020b}
(or the first part of this proof) $\widetilde{v}\in\H{k+3}{}(\rt)$
such that $\symGrad\widetilde{v}=\widetilde{S}$ in $\widehat{\om}$.
Then $S:=\widetilde{S}-\symGrad\widetilde{v}$
belongs to $\H{k+2}{\S}(\rt)$ and satisfies $S|_{\widehat{\om}}=0$.
Thus $S\in\H{k+2}{\S,\gat}(\om)$ and depends continuously on $T$.
Furthermore, $\RotRott S=\RotRott\widetilde{S}=\widetilde{T}$ in $\widetilde{\om}$,
in particular $\RotRott S=T$ in $\om$.
We set $\PotP_{\RotRottS,\gat}^{k}T:=S\in\H{k+2}{\S,\gat}(\om)$.
\item
Let $v\in\H{k}{\gat}(\om)$. By definition, $v$ can be extended through $\gat$
by zero to $\widetilde{\om}$ defining $\widetilde{v}\in\H{k}{}(\widetilde{\om})$.
\cite[Theorem 4.2]{PZ2020b} yields $\widetilde{T}\in\H{k+1}{\S}(\rt)$
such that $\Div\widetilde{T}=\widetilde{v}$ in $\widetilde{\om}$.
As $\widetilde{v}=0$ in $\widehat{\om}$, i.e.,
$\widetilde{T}|_{\widehat{\om}}\in\H{k+1}{\S,0}(\Div,\widehat{\om})$,
we get again by \cite[Theorem 4.2]{PZ2020b}
(or the second part of this proof) $\widetilde{S}\in\H{k+3}{\S}(\rt)$
such that $\RotRott\widetilde{S}=\widetilde{T}$ holds in $\widehat{\om}$.
Then $T:=\widetilde{T}-\RotRott\widetilde{S}$
belongs to $\H{k+1}{\S}(\rt)$ with $T|_{\widehat{\om}}=0$.
Hence $T$ belongs to $\H{k+1}{\S,\gat}(\om)$ and depends continuously on $v$.
Furthermore, $\Div T=\Div\widetilde{T}=\widetilde{v}$ in $\widetilde{\om}$,
in particular $\Div T=v$ in $\om$.
Finally, we define $\PotP_{\DivS,\gat}^{k}v:=T\in\H{k+1}{\S,\gat}(\om)$.
\end{itemize}
The assertion about the compact supports is trivial.
\end{proof}

As a simple consequence of Theorem \ref{highorderregpotextdomela} 
we obtain a few corollaries.

\begin{cor}[regular potentials for extendable domains]
\label{highorderregpotextdomelacor}
Let $(\om,\gat)$ be an extendable bounded strong Lipschitz pair
and let $k\geq0$. Then the regular potentials representations 
\begin{align*}
\bH{k}{\S,\gat,0}(\RotRott,\om)
=\H{k}{\S,\gat,0}(\RotRott,\om)
&=\symGrad\H{k}{\gat}(\symGrad,\om)
=\symGrad\H{k+1}{\gat}(\om)\\
&=R(\symGradgatk),\\
\bH{k}{\S,\gat,0}(\Div,\om)
=\H{k}{\S,\gat,0}(\Div,\om)
&=\RotRott\H{k}{\S,\gat}(\RotRott,\om)
=\RotRott\H{k+2}{\S,\gat}(\om)\\
&=\RotRott\H{k+1,k}{\S,\gat}(\RotRott,\om)\\
&=R(\RotRottSgatk)
=R(\RotRottSgatkpok),\\
\H{k}{\gat}(\om)\cap(\RM_{\gan})^{\bot_{\L{2}{}(\om)}}
&=\Div\H{k}{\S,\gat}(\Div,\om)
=\Div\H{k+1}{\S,\gat}(\om)\\
&=R(\DivSgatk)
\end{align*}
hold, and the potentials can be chosen such that they depend continuously on the data.
In particular, the latter spaces are closed subspaces of 
$\H{k}{\S}(\om)$ and $\H{k}{}(\om)$, respectively.
\end{cor}

\begin{proof}
By Theorem \ref{highorderregpotextdomela} we have 
\begin{align*}
\bH{k}{\S,\gat,0}(\Div,\om)
&=\RotRott\PotP_{\RotRottS,\gat}^{k}\bH{k}{\S,\gat,0}(\Div,\om)
\subset\RotRott\H{k+2}{\S,\gat}(\om)\\
&\subset\RotRott\H{k+1,k}{\S,\gat}(\RotRott,\om)
\subset\RotRott\H{k}{\S,\gat}(\RotRott,\om)\\
&\subset\H{k}{\S,\gat,0}(\Div,\om)
\subset\bH{k}{\S,\gat,0}(\Div,\om).
\end{align*}
The other identities follow analogously.
\end{proof}

\begin{cor}[regular decompositions for extendable domains]
\label{highorderregdecoextdomelacor}
Let $(\om,\gat)$ be an extendable bounded strong Lipschitz pair
and let $k\geq0$. Then the bounded regular decompositions 
\begin{align*}
\bH{k}{\S,\gat}(\RotRott,\om)
&=\H{k+2}{\S,\gat}(\om)
+\symGrad\H{k+1}{\gat}(\om)
=R(\PotP_{\RotRottS,\gat}^{k})
\dotplus\H{k}{\S,\gat,0}(\RotRott,\om)\\
&=R(\PotP_{\RotRottS,\gat}^{k})
\dotplus\symGrad\H{k+1}{\gat}(\om)\\
&=R(\PotP_{\RotRottS,\gat}^{k})
\dotplus\symGrad R(\PotP_{\symGrad,\gat}^{k}),\\
\bH{k}{\S,\gat}(\Div,\om)
&=\H{k+1}{\S,\gat}(\om)
+\RotRott\H{k+2}{\S,\gat}(\om)
=R(\PotP_{\DivS,\gat}^{k})
\dotplus\H{k}{\S,\gat,0}(\Div,\om)\\
&=R(\PotP_{\DivS,\gat}^{k})
\dotplus\RotRott\H{k+2}{\S,\gat}(\om)\\
&=R(\PotP_{\DivS,\gat}^{k})
\dotplus\RotRott R(\PotP_{\RotRottS,\gat}^{k})
\end{align*}
hold with bounded linear regular decomposition operators
\begin{align*}
\PotQ_{\RotRottS,\gat}^{k,1}:=\PotP_{\RotRottS,\gat}^{k}\RotRott:
\bH{k}{\S,\gat}(\RotRott,\om)&\to\H{k+2}{\S,\gat}(\om),\\
\PotQ_{\RotRottS,\gat}^{k,0}:=\PotP_{\symGrad,\gat}^{k}(1-\PotQ_{\RotRottS,\gat}^{k,1}):
\bH{k}{\S,\gat}(\RotRott,\om)&\to\H{k+1}{\gat}(\om),\\
\PotQ_{\DivS,\gat}^{k,1}:=\PotP_{\DivS,\gat}^{k}\Div:
\bH{k}{\S,\gat}(\Div,\om)&\to\H{k+1}{\S,\gat}(\om),\\
\PotQ_{\DivS,\gat}^{k,0}:=\PotP_{\RotRottS,\gat}^{k}(1-\PotQ_{\DivS,\gat}^{k,1}):
\bH{k}{\S,\gat}(\Div,\om)&\to\H{k+2}{\S,\gat}(\om)
\end{align*}
satisfying
\begin{align*}
\PotQ_{\RotRottS,\gat}^{k,1}
+\symGrad\PotQ_{\RotRottS,\gat}^{k,0}
&=\id_{\bH{k}{\S,\gat}(\RotRott,\om)},\\
\PotQ_{\DivS,\gat}^{k,1}
+\RotRott\PotQ_{\DivS,\gat}^{k,0}
&=\id_{\bH{k}{\S,\gat}(\Div,\om)}.
\end{align*}
\end{cor}

\begin{rem}
Note that for (bounded linear) potential operators $\PotP_{\A_{n-1}}$ and $\PotP_{\A_{n}}$ the identity
\begin{align*}
\PotQ_{\A_{n}}^{1}+\A_{n-1}\PotQ_{\A_{n}}^{0}&=\id_{D(\A_{n})}\quad\text{with}
&
\PotQ_{\A_{n}}^{1}:=\PotP_{\A_{n}}\A_{n}:D(\A_{n})&\to\H{+}{n},\\
&&
\PotQ_{\A_{n}}^{0}:=\PotP_{\A_{n-1}}(1-\PotQ_{\A_{n}}^{1}):D(\A_{n})&\to\H{+}{n-1}
\end{align*}
is a general structure of a (bounded) regular decomposition. Moreover:
\begin{itemize}
\item[\bf(i)]
$R(\PotQ_{\A_{n}}^{1})=R(\PotP_{\A_{n}})$ 
and $R(\PotQ_{\A_{n}}^{0})=R(\PotP_{\A_{n-1}})$.
\item[\bf(ii)]
$N(\A_{n})$ is invariant under $\PotQ_{\A_{n}}^{1}$,
as $\A_{n}=\A_{n}\PotQ_{\A_{n}}^{1}$ holds by the complex property.
\item[\bf(iii)]
$\PotQ_{\A_{n}}^{1}$ and $\A_{n-1}\PotQ_{\A_{n}}^{0}=1-\PotQ_{\A_{n}}^{1}$
are projections.
\item[\bf(iv)]
There exists $c>0$ such that for all $x\in D(\A_{n})$
$$\norm{\PotQ_{\A_{n}}^{1}x}_{\H{+}{n}}\leq c\norm{\A_{n}x}_{\H{}{n+1}}.$$
\item[\bf(iv')]
In particular, $\PotQ_{\A_{n}}^{1}|_{N(\A_{n})}=0$.
\end{itemize}
\end{rem}

\begin{cor}[weak and strong partial boundary conditions coincide for extendable domains]
\label{weakstrongextdomelacor}
Let $(\om,\gat)$ be an extendable bounded strong Lipschitz pair
and let $k\geq0$. Then weak and strong boundary conditions coincide, i.e.,
\begin{align*}
\bH{k}{\gat}(\symGrad,\om)
&=\H{k}{\gat}(\symGrad,\om)
=\H{k+1}{\gat}(\om)
=\bH{k+1}{\gat}(\om),\\
\bH{k}{\S,\gat}(\RotRott,\om)
&=\H{k}{\S,\gat}(\RotRott,\om),\\
\bH{k}{\S,\gat}(\Div,\om)
&=\H{k}{\S,\gat}(\Div,\om).
\end{align*}
\end{cor}

\begin{proof}[Proof of Corollary \ref{highorderregdecoextdomelacor} and Corollary \ref{weakstrongextdomelacor}]
Let us pick $S\in\bH{k}{\S,\gat}(\RotRott,\om)$. By Theorem \ref{highorderregpotextdomela} we have 
$\RotRott S\in\bH{k}{\S,\gat,0}(\Div,\om)$ and 
$\widehat{S}:=\PotP_{\RotRottS,\gat}^{k}\RotRott S\in\H{k+2}{\S,\gat}$.
Hence, we obtain $S-\widehat{S}\in\bH{k}{\S,\gat,0}(\RotRott,\om)$
and Theorem \ref{highorderregpotextdomela} shows
$v:=\PotP_{\symGrad,\gat}^{k}(S-\widehat{S})\in\H{k+1}{\gat}(\om)$ 
and thus 
$$S=\widehat{S}+\symGrad v
\in\H{k+2}{\S,\gat}(\om)+\symGrad\H{k+1}{\gat}(\om)
\subset\H{k}{\S,\gat}(\RotRott,\om).$$
For the directness let 
$S=\PotP_{\RotRottS,\gat}^{k}T\in\bH{k}{\S,\gat,0}(\RotRott,\om)$ 
with some $T\in\bH{k}{\S,\gat,0}(\Div,\om)$.
Then $0=\RotRott S=T$ and thus $S=0$.
The assertions about the corresponding $\Div$-spaces follow analogously.
Let $v\in\bH{k}{\gat}(\symGrad,\om)$. Then $\symGrad v\in\bH{k}{\S,\gat,0}(\RotRott,\om)$
and Theorem \ref{highorderregpotextdomela} yields
$\widehat{v}:=\PotP_{\symGrad,\gat}^{k}\symGrad v\in\H{k+1}{\gat}(\om)$.
As $\symGrad(v-\widehat{v})=0$, we have $v-\widehat{v}=:r\in\RM$,
which even vanishes if $\gat\neq\emptyset$. Hence, 
$v=\widehat{v}+r\in\H{k+1}{\gat}(\om)$.
\end{proof}

By similar arguments we also obtain the following (non-standard) 
versions of Corollary \ref{highorderregdecoextdomelacor} and Corollary \ref{weakstrongextdomelacor}.

\begin{cor}[Corollary \ref{highorderregdecoextdomelacor} and Corollary \ref{weakstrongextdomelacor}
for non-standard Sobolev spaces]
\label{highorderregdecoextdomelacornonstandard}
Let $(\om,\gat)$ be an extendable bounded strong Lipschitz pair
and let $k\geq1$. Then the bounded regular decompositions 
\begin{align*}
\bH{k,k-1}{\S,\gat}(\RotRott,\om)
&=\H{k+1}{\S,\gat}(\om)
+\symGrad\H{k+1}{\gat}(\om)
=R(\PotP_{\RotRottS,\gat}^{k-1})
\dotplus\H{k}{\S,\gat,0}(\RotRott,\om)\\
&=R(\PotP_{\RotRottS,\gat}^{k-1})
\dotplus\symGrad\H{k+1}{\gat}(\om)\\
&=R(\PotP_{\RotRottS,\gat}^{k-1})
\dotplus\symGrad R(\PotP_{\symGrad,\gat}^{k})
=\H{k,k-1}{\S,\gat}(\RotRott,\om)
\end{align*}
hold with bounded linear regular decomposition operators
\begin{align*}
\PotQ_{\RotRottS,\gat}^{k,k-1,1}:=\PotP_{\RotRottS,\gat}^{k-1}\RotRott:
\bH{k,k-1}{\S,\gat}(\RotRott,\om)&\to\H{k+1}{\S,\gat}(\om),\\
\PotQ_{\RotRottS,\gat}^{k,k-1,0}:=\PotP_{\symGrad,\gat}^{k}(1-\PotQ_{\RotRottS,\gat}^{k,k-1,1}):
\bH{k,k-1}{\S,\gat}(\RotRott,\om)&\to\H{k+1}{\gat}(\om)
\end{align*}
satisfying
$\PotQ_{\RotRottS,\gat}^{k,k-1,1}
+\symGrad\PotQ_{\RotRottS,\gat}^{k,k-1,0}
=\id_{\bH{k,k-1}{\S,\gat}(\RotRott,\om)}$.
In particular, weak and strong boundary conditions coincide also for the non-standard Sobolev spaces.
\end{cor}

Recall the Hilbert complexes and cohomology groups from Section \ref{sec:elacomplexes} and Section \ref{sec:dirneu}.

\begin{theo}[closed and exact Hilbert complexes for extendable domains]
\label{theo:closedhilcom}
Let $(\om,\gat)$ be an extendable bounded strong Lipschitz pair
and let $k\geq0$. The domain complexes of linear elasticity
\begin{equation*}
\scriptsize
\def\arrowlength{16ex}
\def\arrowdistance{0}
\begin{tikzcd}[column sep=\arrowlength]
\RM_{\gat}
\arrow[r, rightarrow, shift left=\arrowdistance, "\iota_{\RM_{\gat}}"] 
& 
[-3em]
\H{k+1}{\gat}
\ar[r, rightarrow, shift left=\arrowdistance, "\symGradgatk"] 
&
[-1em]
\H{k}{\S,\gat}(\RotRott)
\ar[r, rightarrow, shift left=\arrowdistance, "\RotRottSgatk"] 
& 
[-1em]
\H{k}{\S,\gat}(\Div)
\arrow[r, rightarrow, shift left=\arrowdistance, "\DivSgatk"] 
& 
[-2em]
\H{k}{\gat}
\arrow[r, rightarrow, shift left=\arrowdistance, "\pi_{\RM_{\gan}}"] 
&
[-2em]
\RM_{\gan},
\end{tikzcd}
\end{equation*}
\vspace*{-5mm}
\begin{equation*}
\scriptsize
\def\arrowlength{16ex}
\def\arrowdistance{0}
\begin{tikzcd}[column sep=\arrowlength]
\RM_{\gat}
\arrow[r, leftarrow, shift right=\arrowdistance, "\pi_{\RM_{\gat}}"]
& 
[-3em]
\H{k}{\gat}
\ar[r, leftarrow, shift right=\arrowdistance, "-\DivSgank"]
&
[-2em]
\H{k}{\S,\gan}(\Div)
\ar[r, leftarrow, shift right=\arrowdistance, "\RotRottSgank"]
& 
[-1em]
\H{k}{\S,\gan}(\RotRott)
\arrow[r, leftarrow, shift right=\arrowdistance, "-\symGradgank"]
& 
[0em]
\H{k+1}{\gan}
\arrow[r, leftarrow, shift right=\arrowdistance, "\iota_{\RM_{\gan}}"]
&
[-2em]
\RM_{\gan},
\end{tikzcd}
\end{equation*}
and, for $k\geq1$,
\begin{equation*}
\scriptsize
\def\arrowlength{16ex}
\def\arrowdistance{0}
\begin{tikzcd}[column sep=\arrowlength]
\RM_{\gat}
\arrow[r, rightarrow, shift left=\arrowdistance, "\iota_{\RM_{\gat}}"] 
& 
[-3em]
\H{k+1}{\gat}
\ar[r, rightarrow, shift left=\arrowdistance, "\symGradgatk"] 
&
[-1em]
\H{k,k-1}{\S,\gat}(\RotRott)
\ar[r, rightarrow, shift left=\arrowdistance, "\RotRottSgatkkmo"] 
& 
[0em]
\H{k-1}{\S,\gat}(\Div)
\arrow[r, rightarrow, shift left=\arrowdistance, "\DivSgatkmo"] 
& 
[-2em]
\H{k-1}{\gat}
\arrow[r, rightarrow, shift left=\arrowdistance, "\pi_{\RM_{\gan}}"] 
&
[-2em]
\RM_{\gan},
\end{tikzcd}
\end{equation*}
\vspace*{-5mm}
\begin{equation*}
\scriptsize
\def\arrowlength{16ex}
\def\arrowdistance{0}
\begin{tikzcd}[column sep=\arrowlength]
\RM_{\gat}
\arrow[r, leftarrow, shift right=\arrowdistance, "\pi_{\RM_{\gat}}"]
& 
[-2em]
\H{k-1}{\gat}
\ar[r, leftarrow, shift right=\arrowdistance, "-\DivSgankmo"]
&
[-1em]
\H{k-1}{\S,\gan}(\Div)
\ar[r, leftarrow, shift right=\arrowdistance, "\RotRottSgankkmo"]
& 
[0em]
\H{k,k-1}{\S,\gan}(\RotRott)
\arrow[r, leftarrow, shift right=\arrowdistance, "-\symGradgank"]
& 
[0em]
\H{k+1}{\gan}
\arrow[r, leftarrow, shift right=\arrowdistance, "\iota_{\RM_{\gan}}"]
&
[-2em]
\RM_{\gan}
\end{tikzcd}
\end{equation*}
are exact and closed Hilbert complexes.
In particular, all ranges are closed, 
all cohomology groups (Dirichlet/Neumann fields) are trivial,
and the operators from Theorem \ref{highorderregpotextdomela}
are associated bounded regular potential operators.
\end{theo}

\subsubsection{General Strong Lipschitz Domains}
\label{sec:regpotdecogendom}%

Similar to \cite[Lemma 4.8]{PZ2020b} we get the following.

\begin{lem}[cutting lemma]
\label{lem:cutlem}
Let $\varphi\in\C{\infty}{}(\rt)$ and let $k\geq0$.
\begin{itemize}
\item[\bf(i)]
If $T\in\bH{k}{\S,\gat}(\Div,\om)$, then
$\varphi T\in\bH{k}{\S,\gat}(\Div,\om)$ and 
$\Div(\varphi T)=\varphi\Div T+ T\grad\varphi$ 
holds.
\item[\bf(ii)]
If $k\geq1$ and $S\in\bH{k,k-1}{\S,\gat}(\RotRott,\om)$, then 
$\varphi S\in\bH{k,k-1}{\S,\gat}(\RotRott,\om)$ and 
\begin{align*}
\RotRott(\varphi S)
=\varphi\RotRott S
+2\sym\big((\spn\grad\varphi)\Rot S\big)
+\Psi(\Grad\grad\varphi,S)
\end{align*}
holds with an algebraic operator $\Psi$.
In particular, this holds for $S\in\bH{k}{\S,\gat}(\RotRott,\om)$.
\end{itemize}
\end{lem}

We proceed by showing regular decompositions for the elasticity complexes
extending the results of Corollary \ref{highorderregdecoextdomelacor}
and Corollary \ref{highorderregdecoextdomelacornonstandard}.

\begin{lem}[regular decompositions]
\label{lem:highorderregdecoela}
Let $k\geq0$. Then the bounded regular decompositions
\begin{align*}
\bH{k}{\S,\gat}(\Div,\om)
&=\H{k+1}{\S,\gat}(\om)
+\RotRott\H{k+2}{\S,\gat}(\om),\\
\bH{k}{\S,\gat}(\RotRott,\om)
&=\H{k+2}{\S,\gat}(\om)
+\symGrad\H{k+1}{\gat}(\om)
\intertext{and, for $k\geq1$, the non-standard bounded regular decompositions}
\bH{k}{\S,\gat}(\RotRott,\om)
\subset\bH{k,k-1}{\S,\gat}(\RotRott,\om)
&=\H{k+1}{\S,\gat}(\om)
+\symGrad\H{k+1}{\gat}(\om)
\end{align*}
hold with bounded linear regular decomposition operators
\begin{align*}
\PotQ_{\RotRottS,\gat}^{k,1}:
\bH{k}{\S,\gat}(\RotRott,\om)&\to\H{k+2}{\S,\gat}(\om),
&
\PotQ_{\DivS,\gat}^{k,1}:
\bH{k}{\S,\gat}(\Div,\om)&\to\H{k+1}{\S,\gat}(\om),\\
\PotQ_{\RotRottS,\gat}^{k,0}:
\bH{k}{\S,\gat}(\RotRott,\om)&\to\H{k+1}{\gat}(\om),
&
\PotQ_{\DivS,\gat}^{k,0}:
\bH{k}{\S,\gat}(\Div,\om)&\to\H{k+2}{\S,\gat}(\om),\\
\PotQ_{\RotRottS,\gat}^{k,k-1,1}:
\bH{k,k-1}{\S,\gat}(\RotRott,\om)&\to\H{k+1}{\S,\gat}(\om),\\
\PotQ_{\RotRottS,\gat}^{k,k-1,0}:
\bH{k,k-1}{\S,\gat}(\RotRott,\om)&\to\H{k+1}{\gat}(\om)
\end{align*}
satisfying
\begin{align*}
\PotQ_{\DivS,\gat}^{k,1}
+\RotRott\PotQ_{\DivS,\gat}^{k,0}
&=\id_{\bH{k}{\S,\gat}(\Div,\om)},\\
\PotQ_{\RotRottS,\gat}^{k,1}
+\symGrad\PotQ_{\RotRottS,\gat}^{k,0}
&=\id_{\bH{k}{\S,\gat}(\RotRott,\om)},\\
\PotQ_{\RotRottS,\gat}^{k,k-1,1}
+\symGrad\PotQ_{\RotRottS,\gat}^{k,k-1,0}
&=\id_{\bH{k,k-1}{\S,\gat}(\RotRott,\om)},\qquad
k\geq1.
\end{align*}
It holds $\Div\PotQ_{\DivS,\gat}^{k,1}=\bDivSgatk$
and thus $\bH{k}{\S,\gat,0}(\Div,\om)$ is invariant under $\PotQ_{\DivS,\gat}^{k,1}$.
Analogously, $\RotRott\PotQ_{\RotRottS,\gat}^{k,1}=\bRotRottSgatk$
and $\RotRott\PotQ_{\RotRottS,\gat}^{k,k-1,1}=\bRotRottSgatkkmo$
and thus $\bH{k}{\S,\gat,0}(\RotRott,\om)$ 
is invariant under $\PotQ_{\RotRottS,\gat}^{k,1}$
and $\PotQ_{\RotRottS,\gat}^{k,k-1,1}$, respectively.
\end{lem}

\begin{cor}[weak and strong partial boundary conditions coincide]
\label{cor:weakstrongela}
Let $k\geq0$. Weak and strong boundary conditions coincide, i.e.,
\begin{align*}
\bH{k}{\gat}(\symGrad,\om)
&=\H{k}{\gat}(\symGrad,\om)
=\H{k+1}{\gat}(\om)
=\bH{k+1}{\gat}(\om),\\
\bH{k}{\S,\gat}(\Div,\om)
&=\H{k}{\S,\gat}(\Div,\om),\\
\bH{k}{\S,\gat}(\RotRott,\om)
&=\H{k}{\S,\gat}(\RotRott,\om),\\
\bH{k,k-1}{\S,\gat}(\RotRott,\om)
&=\H{k,k-1}{\S,\gat}(\RotRott,\om),\qquad
k\geq1.
\end{align*}
In particular, $\bsymGradgatk=\symGradgatk$, $\bRotRottSgatk=\RotRottSgatk$,
and $\bDivSgatk=\DivSgatk$, as well as, for $k\geq1$, $\bRotRottSgatkkmo=\RotRottSgatkkmo$.
\end{cor}

\begin{proof}[Proof of Lemma \ref{lem:highorderregdecoela} and Corollary \ref{cor:weakstrongela}.]
According to \cite{PS2021a} and \cite{PZ2020b}, cf.~\cite{BPS2016a,BPS2018a,BPS2019a},
let $(U_{\ell},\varphi_{\ell})$ be a partition of unity for $\om$, i.e.,
$$\om=\bigcup_{\ell=-L}^{L}\om_{\ell},\qquad
\om_{\ell}:=\om\cap U_{\ell},\qquad
\varphi_{\ell}\in\C{\infty}{\p U_{\ell}}(U_{\ell}),$$
and $(\om_{\ell},\widehat\ga_{t,\ell})$ are extendable bounded strong Lipschitz pairs.
Recall $\ga_{t,\ell}:=\gat\cap U_{\ell}$ and $\widehat\ga_{t,\ell}$.
\begin{itemize}
\item
Let $k\geq0$ and let $T\in\bH{k}{\S,\gat}(\Div,\om)$. Then by definition
$T|_{\om_{\ell}}\in\bH{k}{\S,\ga_{t,\ell}}(\Div,\om_{\ell})$ 
and we decompose by Corollary \ref{highorderregdecoextdomelacor}
$$T|_{\om_{\ell}}=T_{\ell,1}
+\RotRott S_{\ell,0}$$
with 
$T_{\ell,1}:=\PotQ_{\DivS,\ga_{t,\ell}}^{k,1}T|_{\om_{\ell}}\in\H{k+1}{\S,\ga_{t,\ell}}(\om_{\ell})$
and 
$S_{\ell,0}:=\PotQ_{\DivS,\ga_{t,\ell}}^{k,0}T|_{\om_{\ell}}\in\H{k+2}{\S,\ga_{t,\ell}}(\om_{\ell})$.
Lemma \ref{lem:cutlem} yields
\begin{align*}
\varphi_{\ell}T|_{\om_{\ell}}
&=\varphi_{\ell}T_{\ell,1}
+\varphi_{\ell}\RotRott S_{\ell,0}\\
&=\overbrace{\varphi_{\ell}T_{\ell,1}
-2\sym\big((\spn\grad\varphi_{\ell})\Rot S_{\ell,0}\big)
-\Psi(\Grad\grad\varphi_{\ell},S_{\ell,0})}^{=:T_{\ell}}\\
&\qquad+\RotRott(\underbrace{\varphi_{\ell}S_{\ell,0})}_{=:S_{\ell}}
\end{align*}
with $T_{\ell}\in\H{k+1}{\S,\widehat\ga_{t,\ell}}(\om_{\ell})$
and $S_{\ell}\in\H{k+2}{\S,\widehat\ga_{t,\ell}}(\om_{\ell})$.
Extending $T_{\ell}$ and $S_{\ell}$ by zero to $\om$
gives tensor fields $\widetilde{T}_{\ell}\in\H{k+1}{\S,\gat}(\om)$ and 
$\widetilde{S}_{\ell}\in\H{k+2}{\S,\gat}(\om)$ as well as 
\begin{align*}
T=\sum_{\ell=-L}^{L}\varphi_{\ell}T|_{\om_{\ell}}
&=\sum_{\ell=-L}^{L}\widetilde{T}_{\ell}+\RotRott\sum_{\ell=-L}^{L}\widetilde{S}_{\ell}\\
&\in\H{k+1}{\S,\gat}(\om)+\RotRott\H{k+2}{\S,\gat}(\om)
\subset\H{k}{\S,\gat}(\Div,\om).
\end{align*}
As all operations have been linear and continuous we set
$$\PotQ_{\DivS,\gat}^{k,1}T:=\sum_{\ell=-L}^{L}\widetilde{T}_{\ell}\in\H{k+1}{\S,\gat}(\om),\qquad
\PotQ_{\DivS,\gat}^{k,0}T:=\sum_{\ell=-L}^{L}\widetilde{S}_{\ell}\in\H{k+2}{\S,\gat}(\om).$$
\item
Let $k\geq1$ and let  $S\in\bH{k,k-1}{\S,\gat}(\RotRott,\om)$. Then by definition
$S|_{\om_{\ell}}\in\bH{k,k-1}{\S,\ga_{t,\ell}}(\RotRott,\om_{\ell})$ 
and we decompose by Corollary \ref{highorderregdecoextdomelacornonstandard}
$$S|_{\om_{\ell}}=S_{\ell,1}
+\symGrad v_{\ell,0}$$
with 
$S_{\ell,1}:=\PotQ_{\RotRottS,\ga_{t,\ell}}^{k,k-1,1}S|_{\om_{\ell}}\in\H{k+1}{\S,\ga_{t,\ell}}(\om_{\ell})$
and 
$v_{\ell,0}:=\PotQ_{\RotRottS,\ga_{t,\ell}}^{k,k-1,0}S|_{\om_{\ell}}\in\H{k+1}{\ga_{t,\ell}}(\om_{\ell})$.
Thus
\begin{align}
\begin{aligned}
\label{decovarphiS}
\varphi_{\ell}S|_{\om_{\ell}}
&=\varphi_{\ell}S_{\ell,1}
+\varphi_{\ell}\symGrad v_{\ell,0}\\
&=\underbrace{\varphi_{\ell}S_{\ell,1}
-\sym\big(v_{\ell,0}(\grad\varphi_{\ell})^{\top}\big)}_{=:S_{\ell}}
+\symGrad(\underbrace{\varphi_{\ell}v_{\ell,0})}_{=:v_{\ell}}
\end{aligned}
\end{align}
with $S_{\ell}\in\H{k+1}{\S,\widehat\ga_{t,\ell}}(\om_{\ell})$
and $v_{\ell}\in\H{k+1}{\widehat\ga_{t,\ell}}(\om_{\ell})$.
Extending $S_{\ell}$ and $v_{\ell}$ by zero to $\om$
gives fields $\widetilde{S}_{\ell}\in\H{k+1}{\S,\gat}(\om)$ and 
$\widetilde{v}_{\ell}\in\H{k+1}{\gat}(\om)$ as well as 
\begin{align*}
S=\sum_{\ell=-L}^{L}\varphi_{\ell}S|_{\om_{\ell}}
&=\sum_{\ell=-L}^{L}\widetilde{S}_{\ell}+\symGrad\sum_{\ell=-L}^{L}\widetilde{v}_{\ell}\\
&\in\H{k+1}{\S,\gat}(\om)+\symGrad\H{k+1}{\gat}(\om)
\subset\H{k,k-1}{\S,\gat}(\RotRott,\om).
\end{align*}
As all operations have been linear and continuous we set
$$\PotQ_{\RotRottS,\gat}^{k,k-1,1}S:=\sum_{\ell=-L}^{L}\widetilde{S}_{\ell}\in\H{k+1}{\S,\gat}(\om),\qquad
\PotQ_{\RotRottS,\gat}^{k,k-1,0}S:=\sum_{\ell=-L}^{L}\widetilde{v}_{\ell}\in\H{k+1}{\gat}(\om).$$
\item
Let $k\geq0$ and let $S\in\bH{k}{\S,\gat}(\RotRott,\om)$. Then by definition
$S|_{\om_{\ell}}\in\bH{k}{\S,\ga_{t,\ell}}(\RotRott,\om_{\ell})$ 
and we decompose by Corollary \ref{highorderregdecoextdomelacor}
$$S|_{\om_{\ell}}=S_{\ell,1}
+\symGrad v_{\ell,0}$$
with 
$S_{\ell,1}:=\PotQ_{\RotRottS,\ga_{t,\ell}}^{k,1}S|_{\om_{\ell}}\in\H{k+2}{\S,\ga_{t,\ell}}(\om_{\ell})$
and 
$v_{\ell,0}:=\PotQ_{\RotRottS,\ga_{t,\ell}}^{k,0}S|_{\om_{\ell}}\in\H{k+1}{\ga_{t,\ell}}(\om_{\ell})$.
Now we follow the arguments from \eqref{decovarphiS} on.
Note that still only $S_{\ell}\in\H{k+1}{\S,\widehat\ga_{t,\ell}}(\om_{\ell})$ holds,
i.e., we have lost one order of regularity for $S_{\ell}$.
Nevertheless, we get 
$$S\in\H{k+1}{\S,\gat}(\om)+\symGrad\H{k+1}{\gat}(\om),$$
and all operations have been linear and continuous.
But this implies by the previous step
$$S\in\H{k+1,k}{\S,\gat}(\RotRott,\om)+\symGrad\H{k+1}{\gat}(\om).$$
Again by the previous step we obtain 
\begin{align*}
S&\in\H{k+2}{\S,\gat}(\om)+\symGrad\H{k+2}{\gat}(\om)+\symGrad\H{k+1}{\gat}(\om)\\
&=\H{k+2}{\S,\gat}(\om)+\symGrad\H{k+1}{\gat}(\om)
\subset\H{k}{\S,\gat}(\RotRott,\om),
\end{align*}
and all operations have been linear and continuous.
\end{itemize}

It remains to prove $\bH{k}{\gat}(\symGrad,\om)\subset\H{k}{\gat}(\symGrad,\om)$.
Let $v\in\bH{k}{\gat}(\symGrad,\om)$. 
Then we have 
$\varphi_{\ell}v\in\bH{k}{\widehat\ga_{t,\ell}}(\symGrad,\om_{\ell})
=\H{k}{\widehat\ga_{t,\ell}}(\symGrad,\om_{\ell})=\H{k+1}{\widehat\ga_{t,\ell}}(\om_{\ell})$
by Corollary \ref{weakstrongextdomelacor}.
Extending $\varphi_{\ell}v$ by zero to $\om$
yields vector fields $v_{\ell}\in\H{k+1}{\gat}(\om)$
as well as $v=\sum_{\ell}\varphi_{\ell}v=\sum_{\ell}v_{\ell}\in\H{k+1}{\gat}(\om)$.
\end{proof}

\subsection{Mini FA-ToolBox}
\label{sec:fatb}%

\subsubsection{Zero Order Mini FA-ToolBox}
\label{sec:zofatb}%

Recall Section \ref{sec:dirneu} and let $\eps$, $\mu$ be admissible.
In Section \ref{sec:elaop} (for $\eps=\mu=\id$) we have seen that
the densely defined and closed linear operators
\begin{align*}
\A_{0}=\symGradgat:\H{1}{\gat}(\om)\subset\L{2}{}(\om)&\to\L{2}{\S,\eps}(\om),\\
\A_{1}=\mu^{-1}\RotRottSgat:\H{}{\S,\gat}(\RotRott\!\!,\om)\subset\L{2}{\S,\eps}(\om)&\to\L{2}{\S,\mu}(\om),\\
\A_{2}=\DivSgat\mu:\mu^{-1}\H{}{\S,\gat}(\Div,\om)\subset\L{2}{\S,\mu}(\om)&\to\L{2}{}(\om),\\
\A_{0}^{*}=-\DivSgan\eps:\eps^{-1}\H{}{\S,\gan}(\Div,\om)\subset\L{2}{\S,\eps}(\om)&\to\L{2}{}(\om),\\
\A_{1}^{*}=\eps^{-1}\RotRottSgan:\H{}{\S,\gan}(\RotRott\!\!,\om)\subset\L{2}{\S,\mu}(\om)&\to\L{2}{\S,\eps}(\om),\\
\A_{2}^{*}=-\symGradgan:\H{1}{\gan}(\om)\subset\L{2}{}(\om)&\to\L{2}{\S,\mu}(\om),
\end{align*}
where we have used Corollary \ref{cor:weakstrongela},
build the long primal and dual elasticity Hilbert complex
\begin{equation}
\label{elacomplex5}
\scriptsize
\def\arrowlength{19ex}
\def\arrowdistance{.8}
\begin{tikzcd}[column sep=\arrowlength]
\RM_{\gat}
\arrow[r, rightarrow, shift left=\arrowdistance, "\A_{-1}=\iota_{\RM_{\gat}}"] 
\arrow[r, leftarrow, shift right=\arrowdistance, "\A_{-1}^{*}=\pi_{\RM_{\gat}}"']
& 
[-2em]
\L{2}{}(\om) 
\ar[r, rightarrow, shift left=\arrowdistance, "\A_{0}=\symGradgat"] 
\ar[r, leftarrow, shift right=\arrowdistance, "\A_{0}^{*}=-\DivSgan\eps"']
&
[-1em]
\L{2}{\S,\eps}(\om) 
\ar[r, rightarrow, shift left=\arrowdistance, "\A_{1}=\mu^{-1}\RotRottSgat"] 
\ar[r, leftarrow, shift right=\arrowdistance, "\A_{1}^{*}=\eps^{-1}\RotRottSgan"']
& 
[0em]
\L{2}{\S,\mu}(\om) 
\arrow[r, rightarrow, shift left=\arrowdistance, "\A_{2}=\DivSgat\mu"] 
\arrow[r, leftarrow, shift right=\arrowdistance, "\A_{2}^{*}=-\symGradgan"']
& 
[-1em]
\L{2}{}(\om) 
\arrow[r, rightarrow, shift left=\arrowdistance, "\A_{3}=\pi_{\RM_{\gan}}"] 
\arrow[r, leftarrow, shift right=\arrowdistance, "\A_{3}^{*}=\iota_{\RM_{\gan}}"']
&
[-2em]
\RM_{\gan}
\end{tikzcd}
\end{equation}
cf. \eqref{elacomplex2}.

\begin{theo}[compact embedding]
\label{theo:cptembzeroorder}
The embedding
$$\H{}{\S,\gat}(\RotRott\!\!,\om)\cap\eps^{-1}\H{}{\S,\gan}(\Div,\om)
\incl\L{2}{\S,\eps}(\om)$$
is compact. Moreover, the compactness does not depend on $\eps$.
\end{theo}

\begin{proof}
Note that this type of compact embedding is independent of $\eps$ and $\mu$,
cf.~\cite[Lemma 5.1]{PW2020a}. So, let $\eps=\mu=\id$.
Lemma \ref{lem:highorderregdecoela} (for $k=0$)
yields the bounded regular decomposition
$$D(\A_{0}^{*})
=\H{}{\S,\gan}(\Div,\om)
=\H{1}{\S,\gan}(\om)
+\RotRott\H{2}{\S,\gan}(\om)
=\H{+}{1}+\A_{1}^{*}\H{+}{2}$$
with $\H{+}{1}=\H{1}{\S,\gan}(\om)$ and $\H{+}{2}=\H{2}{\S,\gan}(\om)$ and 
$\H{}{1}=\H{}{2}=\L{2}{\S}(\om)$.
Rellich's selection theorem and 
\cite[Corollary 2.12]{PZ2020b}, cf.~\cite[Lemma 2.22]{PS2021a},
yield that $D(\A_{1})\cap D(\A_{0}^{*})\incl\H{}{1}$ is compact.
\end{proof}

\begin{rem}[compact embedding]
\label{rem:cptembzeroorder}
The embeddings 
\begin{align*}
D(\A_{0})\cap D(\A_{-1}^{*})
=\H{1}{\gat}(\om)
&\incl\L{2}{}(\om)
=\H{}{0},\\
D(\A_{1})\cap D(\A_{0}^{*})
=\H{}{\S,\gat}(\RotRott\!\!,\om)\cap\eps^{-1}\H{}{\S,\gan}(\Div,\om)
&\incl\L{2}{\S,\eps}(\om)
=\H{}{1},\\
D(\A_{2})\cap D(\A_{1}^{*})
=\mu^{-1}\H{}{\S,\gat}(\Div,\om)\cap\H{}{\S,\gan}(\RotRott\!\!,\om)
&\incl\L{2}{\S,\mu}(\om)
=\H{}{2},\\
D(\A_{3})\cap D(\A_{2}^{*})
=\H{1}{\gan}(\om)
&\incl\L{2}{}(\om)
=\H{}{3}
\end{align*}
are compact, and the compactness does not depend on $\eps$ or $\mu$.
\end{rem}

\begin{theo}[compact elasticity complex]
\label{theo:cptembzeroorder2}
The long primal and dual elasticity Hilbert complex \eqref{elacomplex5} is compact. 
In particular, the complex is closed.
\end{theo}

Let us recall the reduced operators
{\small
\begin{align*}
(\A_{0})_{\bot}=(\symGradgat)_{\bot}:D\big((\symGradgat)_{\bot}\big)\subset
(\RM_{\gat})^{\bot_{\L{2}{}(\om)}}&\to R(\symGradgat),\\
(\A_{1})_{\bot}=(\mu^{-1}\RotRottSgat)_{\bot}:D\big((\mu^{-1}\RotRottSgat)_{\bot}\big)\subset 
N(\mu^{-1}\RotRottSgat)^{\bot_{\L{2}{\S,\eps}(\om)}}&\to R(\mu^{-1}\RotRottSgat),\\
(\A_{2})_{\bot}=(\DivSgat\mu)_{\bot}:D\big((\DivSgat\mu)_{\bot}\big)\subset
N(\DivSgat\mu)^{\bot_{\L{2}{\S,\mu}(\om)}}&\to R(\DivSgat\mu),\\
(\A_{0}^{*})_{\bot}=-(\DivSgan\eps)_{\bot}:D\big((\DivSgan\eps)_{\bot}\big)\subset
N(\DivSgan\eps)^{\bot_{\L{2}{\S,\eps}(\om)}}&\to R(\DivSgan\eps),\\
(\A_{1}^{*})_{\bot}=(\eps^{-1}\RotRottSgan)_{\bot}:D\big((\eps^{-1}\RotRottSgan)_{\bot}\big)\subset 
N(\eps^{-1}\RotRottSgan)^{\bot_{\L{2}{\S,\mu}(\om)}}&\to R(\eps^{-1}\RotRottSgan),\\
(\A_{2}^{*})_{\bot}=(\symGradgan)_{\bot}:D\big((\symGradgan)_{\bot}\big)\subset
(\RM_{\gan})^{\bot_{\L{2}{}(\om)}}&\to R(\symGradgan),\\
\end{align*}
}
with domains of definition 
{\footnotesize
\begin{align*}
D\big((\A_{0})_{\bot}\big)
&=D(\symGradgat)\cap(\RM_{\gat})^{\bot_{\L{2}{}(\om)}},\\
D\big((\A_{1})_{\bot}\big)
&=D(\mu^{-1}\RotRottSgat)\cap N(\mu^{-1}\RotRottSgat)^{\bot_{\L{2}{\S,\eps}(\om)}}
=D(\mu^{-1}\RotRottSgat)\cap R(\eps^{-1}\RotRottSgan),\\
D\big((\A_{2})_{\bot}\big)
&=D(\DivSgat\mu)\cap N(\DivSgat\mu)^{\bot_{\L{2}{\S,\mu}(\om)}}
=D(\DivSgat\mu)\cap R(\symGradgan),\\
D\big((\A_{0}^{*})_{\bot}\big)
&=D(\DivSgan\eps)\cap N(\DivSgan\eps)^{\bot_{\L{2}{\S,\eps}(\om)}}
=D(\DivSgan\eps)\cap R(\symGradgat),\\
D\big((\A_{1}^{*})_{\bot}\big)
&=D(\eps^{-1}\RotRottSgan)\cap N(\eps^{-1}\RotRottSgan)^{\bot_{\L{2}{\S,\mu}(\om)}}
=D(\eps^{-1}\RotRottSgan)\cap R(\mu^{-1}\RotRottSgat),\\
D\big((\A_{2}^{*})_{\bot}\big)
&=D(\symGradgan)\cap(\RM_{\gan})^{\bot_{\L{2}{}(\om)}}.
\end{align*}
}

Note that $R(\A_{n})=R\big((\A_{n})_{\bot}\big)$
and $R(\A_{n}^{*})=R\big((\A_{n}^{*})_{\bot}\big)$ hold.
\cite[Lemma 2.9]{PS2021a} shows:

\begin{theo}[mini FA-ToolBox]
\label{theo:miniFATzeroorder}
For the zero order elasticity complex it holds:
\begin{itemize}
\item[\bf(i)]
The ranges $R(\symGradgat)$, $R(\mu^{-1}\RotRottSgat)$, and $R(\DivSgat\mu)$ are closed.
\item[\bf(i)]
The inverse operators $(\symGradgat)_{\bot}^{-1}$, $(\mu^{-1}\RotRottSgat)_{\bot}^{-1}$, 
and $(\DivSgat\mu)_{\bot}^{-1}$ are compact.
\item[\bf(iii)]
The cohomology group of generalised Dirichlet/Neumann tensor fields
$\Harm{}{\S,\gat,\gan,\eps}(\om)$ is finite-dimensional.
Moreover, the dimension does not depend on $\eps$.
\item[\bf(iv)]
The orthonormal Helmholtz type decompositions
\begin{align*}
\L{2}{\S,\eps}(\om)
&=R(\symGradgat)
\oplus_{\L{2}{\S,\eps}(\om)}N(\DivSgan\eps)\\
&=N(\mu^{-1}\RotRottSgat)
\oplus_{\L{2}{\S,\eps}(\om)}R(\eps^{-1}\RotRottSgan)\\
&=R(\symGradgat)
\oplus_{\L{2}{\S,\eps}(\om)}\Harm{}{\S,\gat,\gan,\eps}(\om)
\oplus_{\L{2}{\S,\eps}(\om)}R(\eps^{-1}\RotRottSgan)
\end{align*}
hold.
\item[\bf(v)]
There exist (optimal) $c_{0},c_{1},c_{2}>0$ such that
the Friedrichs/Poincar\'e type estimates 
\begin{align*}
\forall\,v&\in\H{1}{\gat}(\om)\cap(\RM_{\gat})^{\bot_{\L{2}{}(\om)}}
&
\norm{v}_{\L{2}{}(\om)}&\leq c_{0}\norm{\symGrad v}_{\L{2}{\S,\eps}(\om)},\\
\forall\,T&\in\eps^{-1}\H{}{\S,\gan}(\Div,\om)\cap R(\symGradgat)
&
\norm{T}_{\L{2}{\S,\eps}(\om)}&\leq c_{0}\norm{\Div\eps T}_{\L{2}{}(\om)},\\
\forall\,S&\in\H{}{\S,\gat}(\RotRott\!\!,\om)\cap R(\eps^{-1}\RotRottSgan)
&
\norm{S}_{\L{2}{\S,\eps}(\om)}&\leq c_{1}\norm{\mu^{-1}\RotRott S}_{\L{2}{\S,\mu}(\om)},\\
\forall\,S&\in\H{}{\S,\gan}(\RotRott\!\!,\om)\cap R(\mu^{-1}\RotRottSgat)
&
\norm{S}_{\L{2}{\S,\mu}(\om)}&\leq c_{1}\norm{\eps^{-1}\RotRott S}_{\L{2}{\S,\eps}(\om)},\\
\forall\,T&\in\mu^{-1}\H{}{\S,\gat}(\Div,\om)\cap R(\symGradgan)
&
\norm{T}_{\L{2}{\S,\mu}(\om)}&\leq c_{2}\norm{\Div\mu T}_{\L{2}{}(\om)},\\
\forall\,v&\in\H{1}{\gan}(\om)\cap(\RM_{\gan})^{\bot_{\L{2}{}(\om)}}
&
\norm{v}_{\L{2}{}(\om)}&\leq c_{2}\norm{\symGrad v}_{\L{2}{\S,\mu}(\om)}
\end{align*}
hold.
\item[\bf(vi)]
For all 
$S\in\H{}{\S,\gat}(\RotRott\!\!,\om)
\cap\eps^{-1}\H{}{\S,\gan}(\Div,\om)
\cap\Harm{}{\S,\gat,\gan,\eps}(\om)^{\bot_{\L{2}{\S,\eps}(\om)}}$
it holds
$$\norm{S}_{\L{2}{\S,\eps}(\om)}^2\leq 
c_{1}^2\norm{\mu^{-1}\RotRott S}_{\L{2}{\S,\mu}(\om)}^2
+c_{0}^2\norm{\Div\eps S}_{\L{2}{}(\om)}^2.$$
\item[\bf(vii)]
$\Harm{}{\S,\gat,\gan,\eps}(\om)=\{0\}$, if $(\om,\gat)$ is extendable.
\end{itemize}
\end{theo}

\subsubsection{Higher Order Mini FA-ToolBox}
\label{sec:hofatb}%

For simplicity, let $\eps=\mu=\id$.
From Section \ref{sec:elacomplexes} we recall the 
densely defined and closed higher Sobolev order operators
\begin{align}
\begin{aligned}
\label{defhighorderop}
\symGradgatk:\H{k+1}{\gat}(\om)\subset\H{k}{\gat}(\om)&\to\H{k}{\S,\gat}(\om),\\
\RotRottSgatk:\H{k}{\S,\gat}(\RotRott\!\!,\om)\subset\H{k}{\S,\gat}(\om)&\to\H{k}{\S,\gat}(\om),\\
\RotRottSgatkkmo:\H{k,k-1}{\S,\gat}(\RotRott\!\!,\om)\subset\H{k}{\S,\gat}(\om)&\to\H{k-1}{\S,\gat}(\om),\qquad
k\geq1,\\
\DivSgatk:\H{k}{\S,\gat}(\Div,\om)\subset\H{k}{\S,\gat}(\om)&\to\H{k}{\gat}(\om),
\end{aligned}
\end{align}
building the long elasticity Hilbert complexes
\begin{equation}
\label{elacomplex6}
\scriptsize
\def\arrowlength{15ex}
\def\arrowdistance{.8}
\begin{tikzcd}[column sep=\arrowlength]
\RM_{\gat}
\arrow[r, rightarrow, shift left=\arrowdistance, "\iota_{\RM_{\gat}}"] 
& 
[-2em]
\H{k}{\gat}(\om)
\ar[r, rightarrow, shift left=\arrowdistance, "\symGradgatk"] 
&
[0em]
\H{k}{\S,\gat}(\om)
\ar[r, rightarrow, shift left=\arrowdistance, "\RotRottSgatk"] 
& 
[0em]
\H{k}{\S,\gat}(\om)
\arrow[r, rightarrow, shift left=\arrowdistance, "\DivSgatk"] 
& 
[-2em]
\H{k}{\gat}(\om)
\arrow[r, rightarrow, shift left=\arrowdistance, "\pi_{\RM_{\gan}}"] 
&
[-2em]
\RM_{\gan},\quad k\geq0,
\end{tikzcd}
\end{equation}
\vspace*{-6mm}
\begin{equation}
\label{elacomplex7}
\scriptsize
\def\arrowlength{15ex}
\def\arrowdistance{.8}
\begin{tikzcd}[column sep=\arrowlength]
\RM_{\gat}
\arrow[r, rightarrow, shift left=\arrowdistance, "\iota_{\RM_{\gat}}"] 
& 
[-2em]
\H{k}{\gat}(\om)
\ar[r, rightarrow, shift left=\arrowdistance, "\symGradgatk"] 
&
[0em]
\H{k}{\S,\gat}(\om)
\ar[r, rightarrow, shift left=\arrowdistance, "\RotRottSgatkkmo"] 
& 
[0em]
\H{k-1}{\S,\gat}(\om)
\arrow[r, rightarrow, shift left=\arrowdistance, "\DivSgatkmo"] 
& 
[-2em]
\H{k-1}{\gat}(\om)
\arrow[r, rightarrow, shift left=\arrowdistance, "\pi_{\RM_{\gan}}"] 
&
[-2em]
\RM_{\gan},\quad k\geq1.
\end{tikzcd}
\end{equation}

We start with regular representations
implied by Lemma \ref{lem:highorderregdecoela} and Corollary \ref{cor:weakstrongela}.

\begin{theo}[regular representations and closed ranges]
\label{theo:rangeselawithoutbdpot}
Let $k\geq0$. Then the regular potential representations 
\begin{align*}
R(\symGradgatk)
&=\symGrad\H{k}{\gat}(\symGrad,\om)
=\symGrad\H{k+1}{\gat}(\om)\\
&=\H{k}{\S,\gat}(\om)\cap R(\symGradgat)\\
&=\H{k}{\S,\gat}(\om)\cap\H{}{\S,\gat,0}(\RotRott,\om)\cap\Harm{}{\S,\gat,\gan,\eps}(\om)^{\bot_{\L{2}{\S,\eps}(\om)}}\\
&=\H{k}{\S,\gat,0}(\RotRott,\om)\cap\Harm{}{\S,\gat,\gan,\eps}(\om)^{\bot_{\L{2}{\S,\eps}(\om)}},\\
R(\RotRottSgatkpok)
=R(\RotRottSgatk)
&=\RotRott\H{k}{\S,\gat}(\RotRott,\om)
=\RotRott\H{k+2}{\S,\gat}(\om)\\
&=\RotRott\H{k+1,k}{\S,\gat}(\RotRott,\om)\\
&=\H{k}{\S,\gat}(\om)\cap R(\RotRottSgat)\\
&=\H{k}{\S,\gat}(\om)\cap\H{}{\S,\gat,0}(\Div,\om)\cap\Harm{}{\S,\gan,\gat,\eps}(\om)^{\bot_{\L{2}{\S}(\om)}},\\
&=\H{k}{\S,\gat,0}(\Div,\om)\cap\Harm{}{\S,\gan,\gat,\eps}(\om)^{\bot_{\L{2}{\S}(\om)}},\\
R(\DivSgatk)
&=\Div\H{k}{\S,\gat}(\Div,\om)
=\Div\H{k+1}{\S,\gat}(\om)\\
&=\H{k}{\gat}(\om)\cap R(\DivSgat)
=\H{k}{\gat}(\om)\cap(\RM_{\gan})^{\bot_{\L{2}{}(\om)}}
\end{align*}
hold. In particular, the latter spaces are closed subspaces of 
$\H{k}{\S}(\om)$ and $\H{k}{}(\om)$, respectively,
and all ranges of the higher Sobolev order operators 
in \eqref{defhighorderop} are closed.
Moroever, the long elasticity Hilbert complexes \eqref{elacomplex6} and \eqref{elacomplex7}
are closed.
\end{theo}

Note that in Theorem \ref{theo:rangeselawithoutbdpot} 
we claim nothing about bounded regular potential operators, 
leaving the question of bounded potentials to the next sections,
cf.~Theorem \ref{theo:regpothigherorder}.

\begin{proof}[Proof of Theorem \ref{theo:rangeselawithoutbdpot}.]
We only show the representations for $R(\RotRottSgatk)$.
The other follow analogously, but simpler.
By Lemma \ref{lem:highorderregdecoela} and Corollary \ref{cor:weakstrongela} we have
\begin{align}
\nonumber
\RotRott\H{k+2}{\S,\gat}(\om)
&\subset\RotRott\H{k+1,k}{\S,\gat}(\RotRott,\om)
=R(\RotRottSgatkpok)\\
\nonumber
&\subset\RotRott\H{k}{\S,\gat}(\RotRott,\om)
=R(\RotRottSgatk)
=\RotRott\H{k+2}{\S,\gat}(\om).
\intertext{In particular,}
\label{rangeRotRotkp2}
R(\RotRottSgatk)
&=\RotRott\H{k}{\S,\gat}(\RotRott,\om)
=\RotRott\H{k+2}{\S,\gat}(\om).
\intertext{Moreover,}
\nonumber
R(\RotRottSgatk)
&\subset\H{k}{\S,\gat,0}(\Div,\om)\cap\Harm{}{\S,\gan,\gat,\eps}(\om)^{\bot_{\L{2}{\S}(\om)}}\\
\nonumber
&=\H{k}{\S,\gat}(\om)\cap\H{}{\S,\gat,0}(\Div,\om)\cap\Harm{}{\S,\gan,\gat,\eps}(\om)^{\bot_{\L{2}{\S}(\om)}}
=\H{k}{\S,\gat}(\om)\cap R(\RotRottSgat),
\end{align}
since by Theorem \ref{theo:miniFATzeroorder} (iv)
\begin{align}
\label{helmdecoela1}
R(\RotRottSgat)
=\H{}{\S,\gat,0}(\Div,\om)\cap\Harm{}{\S,\gan,\gat,\eps}(\om)^{\bot_{\L{2}{\S}(\om)}}.
\end{align}
Thus it remains to show
$$\H{k}{\S,\gat,0}(\Div,\om)\cap\Harm{}{\S,\gan,\gat,\eps}(\om)^{\bot_{\L{2}{\S}(\om)}}
\subset\RotRott\H{k}{\S,\gat}(\RotRott,\om),\qquad
k\geq1.$$
For this, let $k\geq1$ and
$T\in\H{k}{\S,\gat,0}(\Div,\om)\cap\Harm{}{\S,\gan,\gat,\eps}(\om)^{\bot_{\L{2}{\S}(\om)}}$.
By \eqref{helmdecoela1} and \eqref{rangeRotRotkp2} we have 
$$T\in R(\RotRottSgat)=\RotRott\H{2}{\S,\gat}(\om)$$
and hence there is $S_{1}\in\H{2}{\S,\gat}(\om)$ such that $\RotRott S_{1}=T$.
We see $S_{1}\in\H{2}{\S,\gat}(\RotRott,\om)$ resp. $S_{1}\in\H{1}{\S,\gat}(\RotRott,\om)$ if $k=1$.
Hence we are done for $k=1$ and $k=2$.
For $k\geq2$ we have 
$T\in\RotRott\H{2}{\S,\gat}(\RotRott,\om)=\RotRott\H{4}{\S,\gat}(\om)$ 
by \eqref{rangeRotRotkp2}.
Thus there is $S_{2}\in\H{4}{\S,\gat}(\om)$ such that $\RotRott S_{2}=T$.
Then $S_{2}\in\H{4}{\S,\gat}(\RotRott,\om)$ resp. $S_{2}\in\H{3}{\S,\gat}(\RotRott,\om)$ if $k=3$,
and we are done for $k=3$ and $k=4$.
After finitely many steps, we observe that
$T$ belongs to $\RotRott\H{k}{\S,\gat}(\RotRott,\om)$,
finishing the proof.
\end{proof}

The reduced operators corresponding to \eqref{defhighorderop} are
\begin{align*}
(\symGradgatk)_{\bot}:D\big((\symGradgatk)_{\bot}\big)\subset
(\RM_{\gat})^{\bot_{\H{k}{\gat}(\om)}}&\to R(\symGradgatk),\\
(\RotRottSgatk)_{\bot}:D\big((\RotRottSgatk)_{\bot}\big)\subset 
N(\RotRottSgatk)^{\bot_{\H{k}{\S,\gat}(\om)}}&\to R(\RotRottSgatk),\\
(\RotRottSgatkkmo)_{\bot}:D\big((\RotRottSgatkkmo)_{\bot}\big)\subset 
N(\RotRottSgatk)^{\bot_{\H{k}{\S,\gat}(\om)}}&\to R(\RotRottSgatkmo),\quad
k\geq1,\\
(\DivSgatk)_{\bot}:D\big((\DivSgatk)_{\bot}\big)\subset
N(\DivSgatk)^{\bot_{\H{k}{\S,\gat}(\om)}}&\to R(\DivSgatk)
\end{align*}
with domains of definition 
\begin{align*}
D\big((\symGradgatk)_{\bot}\big)
&=D(\symGradgatk)\cap(\RM_{\gat})^{\bot_{\H{k}{\gat}(\om)}},\\
D\big((\RotRottSgatk)_{\bot}\big)
&=D(\RotRottSgatk)\cap N(\RotRottSgatk)^{\bot_{\H{k}{\S,\gat}(\om)}},\\
D\big((\RotRottSgatkkmo)_{\bot}\big)
&=D(\RotRottSgatkkmo)\cap N(\RotRottSgatk)^{\bot_{\H{k}{\S,\gat}(\om)}},\quad
k\geq1,\\
D\big((\DivSgatk)_{\bot}\big)
&=D(\DivSgatk)\cap N(\DivSgatk)^{\bot_{\H{k}{\S,\gat}(\om)}}.
\end{align*}

\cite[Lemma 2.1]{PS2021a} and Theorem \ref{theo:rangeselawithoutbdpot} yield:

\begin{theo}[closed ranges and bounded inverse operators]
\label{theo:clrangebdinvela}
Let $k\geq0$. Then:
\begin{itemize}
\item[\bf(i)]
$R(\symGradgatk)=R\big((\symGradgatk)_{\bot}\big)$ are closed and, equivalently, 
the inverse operator 
\begin{align*}
(\symGradgatk)_{\bot}^{-1}:R(\symGradgatk)&\to D\big((\symGradgatk)_{\bot}\big)\\
\text{resp.}\qquad
(\symGradgatk)_{\bot}^{-1}:R(\symGradgatk)&\to D(\symGradgatk)
\end{align*}
is bounded. Equivalently, there is $c>0$ such that
for all $v\in D\big((\symGradgatk)_{\bot}\big)$
$$\norm{v}_{\H{k}{}(\om)}\leq c\norm{\symGrad v}_{\H{k}{\S}(\om)}.$$
\item[\bf(ii)]
$R\big((\RotRottSgatk)_{\bot}\big)
=R(\RotRottSgatk)
=R(\RotRottSgatkpok)
=R\big((\RotRottSgatkpok)_{\bot}\big)$ 
are closed and, equivalently, the inverse operators
\begin{align*}
(\RotRottSgatk)_{\bot}^{-1}:R(\RotRottSgatk)&\to D\big((\RotRottSgatk)_{\bot}\big)\\
\text{resp.}\qquad
(\RotRottSgatk)_{\bot}^{-1}:R(\RotRottSgatk)&\to D(\RotRottSgatk),\\
(\RotRottSgatkpok)_{\bot}^{-1}:R(\RotRottSgatk)&\to D\big((\RotRottSgatkpok)_{\bot}\big)\\
\text{resp.}\qquad
(\RotRottSgatkpok)_{\bot}^{-1}:R(\RotRottSgatk)&\to D(\RotRottSgatkpok)
\end{align*}
are bounded. Equivalently, there is $c>0$ such that for all
$S\in D\big((\RotRottSgatk)_{\bot}\big)$ resp. $S\in D\big((\RotRottSgatkpok)_{\bot}\big)$
$$\norm{S}_{\H{k}{\S}(\om)}\leq c\norm{\RotRott S}_{\H{k}{\S}(\om)}
\quad\text{resp.}\quad
\norm{S}_{\H{k+1}{\S}(\om)}\leq c\norm{\RotRott S}_{\H{k}{\S}(\om)}.$$
\item[\bf(iii)]
$R(\DivSgatk)=R\big((\DivSgatk)_{\bot}\big)$ are closed and, equivalently, 
the inverse operator 
\begin{align*}
(\DivSgatk)_{\bot}^{-1}:R(\DivSgatk)&\to D\big((\DivSgatk)_{\bot}\big)\\
\text{resp.}\qquad
(\DivSgatk)_{\bot}^{-1}:R(\DivSgatk)&\to D(\DivSgatk)
\end{align*}
is bounded. Equivalently, there is $c>0$ such that for all
$T\in D\big((\DivSgatk)_{\bot}\big)$
$$\norm{T}_{\H{k}{\S}(\om)}\leq c\norm{\Div T}_{\H{k}{}(\om)}.$$
\end{itemize}
\end{theo}

\begin{lem}[Schwarz' lemma]
\label{lem:scharzlemma}
Let $0\leq|\alpha|\leq k$.
\begin{itemize}
\item[\bf(i)]
For $S\in\H{k}{\S,\gat}(\RotRott\!\!,\om)$ resp.
$S\in\H{k+1,k}{\S,\gat}(\RotRott\!\!,\om)$ it holds 
$\p^{\alpha}S\in\H{}{\S,\gat}(\RotRott\!\!,\om)$ resp.
$\p^{\alpha}S\in\H{1,0}{\S,\gat}(\RotRott\!\!,\om)$
and $\RotRott\p^{\alpha}S=\p^{\alpha}\RotRott S$.
\item[\bf(ii)]
For $T\in\H{k}{\S,\gat}(\Div,\om)$ it holds 
$\p^{\alpha}T\in\H{}{\S,\gat}(\Div,\om)$ and $\Div\p^{\alpha}T=\p^{\alpha}\Div T$.
\end{itemize}
\end{lem}

\begin{proof}
Let $S\in\H{k}{\S,\gat}(\RotRott\!\!,\om)$. For $\Phi\in\C{\infty}{\gan}(\om)$ we have 
\begin{align*}
\scp{\p^{\alpha}S}{\RotRott\Phi}_{\L{2}{\S}(\om)}
&=(-1)^{|\alpha|}\scp{S}{\RotRott\p^{\alpha}\Phi}_{\L{2}{\S}(\om)}\\
&=(-1)^{|\alpha|}\scp{\RotRott S}{\p^{\alpha}\Phi}_{\L{2}{\S}(\om)}
=\scp{\p^{\alpha}\RotRott S}{\Phi}_{\L{2}{\S}(\om)}
\end{align*}
as $S\in\H{k}{\S,\gat}(\om)\cap\H{}{\S,\gat}(\RotRott,\om)$ and
$\RotRott S\in\H{k}{\S,\gat}(\om)$. Hence
$$\p^{\alpha}S\in\bH{}{\S,\gat}(\RotRott\!\!,\om)
=\H{}{\S,\gat}(\RotRott\!\!,\om)$$
by Corollary \ref{cor:weakstrongela} and 
$\RotRott\p^{\alpha}S=\p^{\alpha}\RotRott S$.
The other assertions follow analogously.
\end{proof}

\begin{theo}[compact embedding]
\label{theo:cptembhigherorder}
Let $k\geq0$. Then the embedding
$$\H{k}{\S,\gat}(\RotRott\!\!,\om)\cap\H{k}{\S,\gan}(\Div,\om)
\incl\H{k}{\S,\ga}(\om)$$
is compact.
\end{theo}

\begin{proof}
We follow in close lines the proof of \cite[Theorem 4.11]{PZ2020b},
cf.~\cite[Theorem 4.16]{PS2021a}, using induction.
The case $k=0$ is given by Theorem \ref{theo:cptembzeroorder}.
Let $k\geq1$ and let $(S_{\ell})$ be a bounded sequence in 
$\H{k}{\S,\gat}(\RotRott\!\!,\om)\cap\H{k}{\S,\gan}(\Div,\om)$.
Note that 
$$\H{k}{\S,\gat}(\RotRott\!\!,\om)\cap\H{k}{\S,\gan}(\Div,\om)
\subset\H{k}{\S,\gat}(\om)\cap\H{k}{\S,\gan}(\om)
=\H{k}{\S,\ga}(\om).$$
By assumption and w.l.o.g. we have that $(S_{\ell})$ 
is a Cauchy sequence in $\H{k-1}{\S,\ga}(\om)$.
Moreover, for all $|\alpha|=k$ we have
$\p^{\alpha}S_{\ell}\in\H{}{\S,\gat}(\RotRott\!\!,\om)\cap\H{}{\S,\gan}(\Div,\om)$
with $\RotRott\p^{\alpha}S_{\ell}=\p^{\alpha}\RotRott S_{\ell}$
and $\Div\p^{\alpha}S_{\ell}=\p^{\alpha}\Div S_{\ell}$
by Lemma \ref{lem:scharzlemma}.
Hence $(\p^{\alpha}S_{\ell})$ is a bounded sequence in the zero order space
$\H{}{\S,\gat}(\RotRott\!\!,\om)\cap\H{}{\S,\gan}(\Div,\om)$.
Thus, w.l.o.g. $(\p^{\alpha}S_{\ell})$ is a Cauchy sequence in $\L{2}{\S}(\om)$
by Theorem \ref{theo:cptembzeroorder}.
Finally, $(S_{\ell})$ is a Cauchy sequence in $\H{k}{\S,\ga}(\om)$,
finishing the proof.
\end{proof}

\begin{rem}[compact embedding]
\label{rem:cptembhigherorder}
For $k\geq1$, cf.~\cite[Remark 4.12]{PZ2020b}, there is
another and slightly more general proof using a variant of \cite[Lemma 2.22]{PS2021a}. 

For this, let $(S_{\ell})$ be a bounded sequence in 
$\H{k}{\S,\gat}(\RotRott\!\!,\om)\cap\H{k}{\S,\gan}(\Div,\om)$.
In particular, $(S_{\ell})$ is bounded in 
$\H{k,k-1}{\S,\gat}(\RotRott\!\!,\om)\cap\H{k}{\S,\gan}(\Div,\om)$.
According to Lemma \ref{lem:highorderregdecoela} we decompose 
$S_{\ell}=T_{\ell}+\symGrad v_{\ell}$
with $T_{\ell}\in\H{k+1}{\S,\gat}(\om)$ and $v_{\ell}\in\H{k+1}{\gat}(\om)$.
By the boundedness of the regular decomposition operators,
$(T_{\ell})$ and $(v_{\ell})$ are bounded in 
$\H{k+1}{\S,\gat}(\om)$ and $\H{k+1}{\gat}(\om)$, respectively.
W.l.o.g. $(T_{\ell})$ and $(v_{\ell})$ converge in 
$\H{k}{\S,\gat}(\om)$ and $\H{k}{\gat}(\om)$, respectively.
For all $0\leq|\alpha|\leq k$ Lemma \ref{lem:scharzlemma} yields
$(\p^{\alpha}S_{\ell})\subset\H{}{\S,\gan}(\Div,\om)$
and $\Div\p^{\alpha}T=\p^{\alpha}\Div T$. With 
$S_{\ell,l}:=S_{\ell}-S_{l}$, $T_{\ell,l}:=T_{\ell}-T_{l}$, 
and $v_{\ell,l}:=v_{\ell}-v_{l}$ we get
\begin{align*}
\norm{S_{\ell,l}}_{\H{k}{\S}(\om)}^2
&=\scp{S_{\ell,l}}{T_{\ell,l}}_{\H{k}{\S}(\om)}
+\scp{S_{\ell,l}}{\symGrad v_{\ell,l}}_{\H{k}{\S}(\om)}\\
&=\scp{S_{\ell,l}}{T_{\ell,l}}_{\H{k}{\S}(\om)}
-\scp{\Div S_{\ell,l}}{v_{\ell,l}}_{\H{k}{}(\om)}
\leq c\big(\norm{T_{\ell,l}}_{\H{k}{\S}(\om)}+\norm{v_{\ell,l}}_{\H{k}{}(\om)}\big)
\to0.
\end{align*}
\end{rem}

The latter remark shows immediately:

\begin{theo}[compact embedding]
\label{theo:cptembhigherorderkkmo}
Let $k\geq1$. Then the embedding
$$\H{k,k-1}{\S,\gat}(\RotRott\!\!,\om)\cap\H{k}{\S,\gan}(\Div,\om)
\incl\H{k}{\S,\ga}(\om)$$
is compact.
\end{theo}

\begin{theo}[Friedrichs/Poincar\'e type estimate]
\label{theo:friedrichspoincarehigherorder}
There exists $\widehat{c}_{k}>0$ such that for all 
$S$ in $\H{k}{\S,\gat}(\RotRott\!\!,\om)
\cap\H{k}{\S,\gan}(\Div,\om)
\cap\Harm{}{\S,\gat,\gan,\id}(\om)^{\bot_{\L{2}{\S}(\om)}}$
$$\norm{S}_{\H{k}{\S}(\om)}\leq 
\widehat{c}_{k}\big(\norm{\RotRott S}_{\H{k}{\S}(\om)}
+\norm{\Div S}_{\H{k}{}(\om)}\big).$$
The condition $\Harm{}{\S,\gat,\gan,\id}(\om)^{\bot_{\L{2}{\S}(\om)}}$
can be replaced by the weaker conditions 
$\Harm{k}{\S,\gat,\gan,\id}(\om)^{\bot_{\L{2}{\S}(\om)}}$ or
$\Harm{k}{\S,\gat,\gan,\id}(\om)^{\bot_{\H{k}{\S}(\om)}}$.
In particular, it holds
\begin{align*}
\forall\,S&\in\H{k}{\S,\gat}(\RotRott\!\!,\om)
\cap R(\RotRottSgank)
&\norm{S}_{\H{k}{\S}(\om)}
&\leq\widehat{c}_{k}\norm{\RotRott S}_{\H{k}{\S}(\om)},\\
\forall\,S&\in\H{k}{\S,\gan}(\Div,\om)
\cap R(\symGradgatk)
&\norm{S}_{\H{k}{\S}(\om)}
&\leq\widehat{c}_{k}\norm{\Div S}_{\H{k}{}(\om)}
\end{align*}
with
\begin{align*}
R(\RotRottSgankpok)=R(\RotRottSgank)
&=\H{k}{\S,\gan,0}(\Div,\om)\cap\Harm{}{\S,\gat,\gan,\id}(\om)^{\bot_{\L{2}{\S}(\om)}},\\
R(\symGradgatk)&=\H{k}{\S,\gat,0}(\RotRott,\om)\cap\Harm{}{\S,\gat,\gan,\id}(\om)^{\bot_{\L{2}{\S}(\om)}}.
\end{align*}

Analogously, for $k\geq1$ there exists $\widehat{c}_{k,k-1}>0$ such that 
$$\norm{S}_{\H{k}{\S}(\om)}\leq 
\widehat{c}_{k,k-1}\big(\norm{\RotRott S}_{\H{k-1}{\S}(\om)}
+\norm{\Div S}_{\H{k}{}(\om)}\big)$$
for all $S$ in $\H{k,k-1}{\S,\gat}(\RotRott\!\!,\om)
\cap\H{k}{\S,\gan}(\Div,\om)
\cap\Harm{}{\S,\gat,\gan,\id}(\om)^{\bot_{\L{2}{\S}(\om)}}$.
Moreover, 
\begin{align*}
\forall\,S&\in\H{k,k-1}{\S,\gat}(\RotRott\!\!,\om)
\cap R(\RotRottSgank)
&\norm{S}_{\H{k}{\S}(\om)}
&\leq\widehat{c}_{k,k-1}\norm{\RotRott S}_{\H{k-1}{\S}(\om)}.
\end{align*}
\end{theo}

\begin{proof}
We follow the proof of \cite[Theorem 4.17]{PS2021a}.
To show the first estimate, we use a standard strategy and assume the contrary. Then there is a sequence 
$$(S_{\ell})\subset\H{k}{\S,\gat}(\RotRott\!\!,\om)
\cap\H{k}{\S,\gan}(\Div,\om)
\cap\Harm{}{\S,\gat,\gan,\id}(\om)^{\bot_{\L{2}{\S}(\om)}}$$
with $\norm{S_{\ell}}_{\H{k}{\S}(\om)}=1$ and 
$\norm{\RotRott S_{\ell}}_{\H{k}{\S}(\om)}
+\norm{\Div S_{\ell}}_{\H{k}{}(\om)}\to0$.
Hence we may assume that $S_{\ell}$ converges weakly to some $S$ in
$\H{k}{\S}(\om)\cap\Harm{}{\S,\gat,\gan,\id}(\om)\cap\Harm{}{\S,\gat,\gan,\id}(\om)^{\bot_{\L{2}{\S}(\om)}}$.
Thus $S=0$. By Theorem \ref{theo:cptembhigherorder} $(S_{\ell})$ converges strongly to $0$ in $\H{k}{\S}(\om)$,
in contradiction to $\norm{S_{\ell}}_{\H{k}{}(\om)}=1$.
The other estimates follow analogously 
resp. with Theorem \ref{theo:rangeselawithoutbdpot} by restriction.
\end{proof}

\begin{rem}[Friedrichs/Poincar\'e/Korn type estimate]
\label{rem:friedrichspoincarekornhigherorder}
Similar to Theorem \ref{theo:friedrichspoincarehigherorder}
and by Rellich's selection theorem
there exists $c>0$ such that for all 
$v\in\H{k+1}{\gat}(\om)\cap(\RM_{\gat})^{\bot_{\L{2}{}(\om)}}$
$$\norm{v}_{\H{k}{}(\om)}\leq c\norm{\symGrad v}_{\H{k}{\S}(\om)}.$$
As in Theorem \ref{theo:clrangebdinvela},
$(\RM_{\gat})^{\bot_{\L{2}{}(\om)}}$ can be replaced 
by $(\RM_{\gat})^{\bot_{\H{k}{\gat}(\om)}}$.
\end{rem}

\subsection{Regular Potentials and Decompositions II}
\label{sec:regpotdeco2}%

Let $k\geq0$.
According to Theorem \ref{theo:clrangebdinvela} the inverses
of the reduced operators  
\begin{align*}
(\symGradgatk)_{\bot}^{-1}:R(\symGradgatk)&\to D(\symGradgatk)=\H{k+1}{\gat}(\om),\\
(\RotRottSgatk)_{\bot}^{-1}:R(\RotRottSgatk)&\to D(\RotRottSgatk)=\H{k}{\S,\gat}(\RotRott,\om),\\
(\RotRottSgatkpok)_{\bot}^{-1}:R(\RotRottSgatk)&\to D(\RotRottSgatkpok)=\H{k+1,k}{\S,\gat}(\RotRott,\om),\\
(\DivSgatk)_{\bot}^{-1}:R(\DivSgatk)&\to D(\DivSgatk)=\H{k}{\S,\gat}(\Div,\om)
\end{align*}
are bounded and we recall the bounded linear regular decomposition operators
\begin{align*}
\PotQ_{\DivS,\gat}^{k,1}:
\H{k}{\S,\gat}(\Div,\om)&\to\H{k+1}{\S,\gat}(\om),
&
\PotQ_{\DivS,\gat}^{k,0}:
\H{k}{\S,\gat}(\Div,\om)&\to\H{k+2}{\S,\gat}(\om),\\
\PotQ_{\RotRottS,\gat}^{k,1}:
\H{k}{\S,\gat}(\RotRott,\om)&\to\H{k+2}{\S,\gat}(\om),
&
\PotQ_{\RotRottS,\gat}^{k,0}:
\H{k}{\S,\gat}(\RotRott,\om)&\to\H{k+1}{\gat}(\om),\\
\PotQ_{\RotRottS,\gat}^{k+1,k,1}:
\H{k+1,k}{\S,\gat}(\RotRott,\om)&\to\H{k+2}{\S,\gat}(\om),
&
\PotQ_{\RotRottS,\gat}^{k+1,k,0}:
\H{k+1,k}{\S,\gat}(\RotRott,\om)&\to\H{k+2}{\gat}(\om)
\end{align*}
from Lemma \ref{lem:highorderregdecoela}.
Similar to \cite[Theorem 4.18, Theorem 5.2]{PS2021a},
cf.~\cite[Lemma 2.22, Theorem 2.23]{PS2021a},
we obtain the following sequence of results:

\begin{theo}[bounded regular potentials from bounded regular decompositions]
\label{theo:regpothigherorder}
For $k\geq0$ there exist bounded linear regular potential operators
\begin{align*}
\PotP_{\symGrad,\gat}^{k}:=(\symGradgatk)_{\bot}^{-1}:
\H{k}{\S,\gat,0}(\RotRott,\om)\cap\Harm{}{\S,\gat,\gan,\eps}(\om)^{\bot_{\L{2}{\S,\eps}(\om)}}
&\to\H{k+1}{\gat}(\om),\\
\PotP_{\RotRottS,\gat}^{k}:=\PotQ_{\RotRottS,\gat}^{k,1}(\RotRottSgatk)_{\bot}^{-1}:
\H{k}{\S,\gat,0}(\Div,\om)\cap\Harm{}{\S,\gan,\gat,\eps}(\om)^{\bot_{\L{2}{\S}(\om)}}
&\to\H{k+2}{\S,\gat}(\om),\\
\PotP_{\RotRottS,\gat}^{k+1,k}:=\PotQ_{\RotRottS,\gat}^{k+1,k,1}(\RotRottSgatkpok)_{\bot}^{-1}:
\H{k}{\S,\gat,0}(\Div,\om)\cap\Harm{}{\S,\gan,\gat,\eps}(\om)^{\bot_{\L{2}{\S}(\om)}}
&\to\H{k+2}{\S,\gat}(\om),\\
\PotP_{\DivS,\gat}^{k}:=\PotQ_{\DivS,\gat}^{k,1}(\DivSgatk)_{\bot}^{-1}:
\H{k}{\gat}(\om)\cap(\RM_{\gan})^{\bot_{\L{2}{}(\om)}}
&\to\H{k+1}{\S,\gat}(\om),
\end{align*}
such that 
\begin{align*}
\symGrad\PotP_{\symGrad,\gat}^{k}
&=\id|_{\H{k}{\S,\gat,0}(\RotRott,\om)\cap\Harm{}{\S,\gat,\gan,\eps}(\om)^{\bot_{\L{2}{\S,\eps}(\om)}}},\\
\RotRott\PotP_{\RotRottS,\gat}^{k+1,k}
=\RotRott\PotP_{\RotRottS,\gat}^{k}
&=\id|_{\H{k}{\S,\gat,0}(\Div,\om)\cap\Harm{}{\S,\gan,\gat,\eps}(\om)^{\bot_{\L{2}{\S}(\om)}}},\\
\Div\PotP_{\DivS,\gat}^{k}
&=\id|_{\H{k}{\gat}(\om)\cap(\RM_{\gan})^{\bot_{\L{2}{}(\om)}}}.
\end{align*}
In particular, all potentials in Theorem \ref{theo:rangeselawithoutbdpot}
can be chosen such that they depend continuously on the data.
$\PotP_{\symGrad,\gat}^{k}$, $\PotP_{\RotRottS,\gat}^{k}$, $\PotP_{\RotRottS,\gat}^{k+1,k}$, and $\PotP_{\DivS,\gat}^{k}$
are right inverses of $\symGrad$, $\RotRott$, and $\Div$, respectively.
\end{theo}

\begin{theo}[bounded regular decompositions from bounded regular potentials]
\label{theo:regdecohigherorder}
For $k\geq0$ the bounded regular decompositions
\begin{align*}
\H{k}{\S,\gat}(\Div,\om)
&=\H{k+1}{\S,\gat}(\om)
+\H{k}{\S,\gat,0}(\Div,\om)
=\H{k+1}{\S,\gat}(\om)
+\RotRott\H{k+2}{\S,\gat}(\om)\\
&=R(\widetilde\PotQ_{\DivS,\gat}^{k,1})
\dotplus\H{k}{\S,\gat,0}(\Div,\om)\\
&=R(\widetilde\PotQ_{\DivS,\gat}^{k,1})
\dotplus R(\widetilde\PotN_{\DivS,\gat}^{k}),\\
\H{k}{\S,\gat}(\RotRott,\om)
&=\H{k+2}{\S,\gat}(\om)
+\H{k}{\S,\gat,0}(\RotRott,\om)
=\H{k+2}{\S,\gat}(\om)
+\symGrad\H{k+1}{\gat}(\om)\\
&=R(\widetilde\PotQ_{\RotRottS,\gat}^{k,1})
\dotplus\H{k}{\S,\gat,0}(\RotRott,\om)\\
&=R(\widetilde\PotQ_{\RotRottS,\gat}^{k,1})
\dotplus R(\widetilde\PotN_{\RotRottS,\gat}^{k}),\\
\H{k+1,k}{\S,\gat}(\RotRott,\om)
&=\H{k+2}{\S,\gat}(\om)
+\H{k+1}{\S,\gat,0}(\RotRott,\om)
=\H{k+2}{\S,\gat}(\om)
+\symGrad\H{k+2}{\gat}(\om)\\
&=R(\widetilde\PotQ_{\RotRottS,\gat}^{k+1,k,1})
\dotplus\H{k+1}{\S,\gat,0}(\RotRott,\om)\\
&=R(\widetilde\PotQ_{\RotRottS,\gat}^{k+1,k,1})
\dotplus R(\widetilde\PotN_{\RotRottS,\gat}^{k+1,k})
\end{align*}
hold with bounded linear regular decomposition operators
\begin{align*}
\widetilde\PotQ_{\DivS,\gat}^{k,1}:=\PotP_{\DivS,\gat}^{k}\DivSgatk:
\H{k}{\S,\gat}(\Div,\om)&\to\H{k+1}{\S,\gat}(\om),\\
\widetilde\PotQ_{\RotRottS,\gat}^{k,1}:=\PotP_{\RotRottS,\gat}^{k}\RotRottSgatk:
\H{k}{\S,\gat}(\RotRott,\om)&\to\H{k+2}{\S,\gat}(\om),\\
\widetilde\PotQ_{\RotRottS,\gat}^{k+1,k,1}:=\PotP_{\RotRottS,\gat}^{k+1,k}\RotRottSgatkpok:
\H{k+1,k}{\S,\gat}(\RotRott,\om)&\to\H{k+2}{\S,\gat}(\om),\\
\widetilde\PotN_{\DivS,\gat}^{k}:
\H{k}{\S,\gat}(\Div,\om)&\to\H{k}{\S,\gat,0}(\Div,\om),\\
\widetilde\PotN_{\RotRottS,\gat}^{k}:
\H{k}{\S,\gat}(\RotRott,\om)&\to\H{k}{\S,\gat,0}(\RotRott,\om),\\
\widetilde\PotN_{\RotRottS,\gat}^{k+1,k}:
\H{k+1,k}{\S,\gat}(\RotRott,\om)&\to\H{k+1}{\S,\gat,0}(\RotRott,\om)
\end{align*}
satisfying
\begin{align*}
\id_{\H{k}{\S,\gat}(\Div,\om)}
&=\widetilde\PotQ_{\DivS,\gat}^{k,1}
+\widetilde\PotN_{\DivS,\gat}^{k},\\
\id_{\H{k}{\S,\gat}(\RotRott,\om)}
&=\widetilde\PotQ_{\RotRottS,\gat}^{k,1}
+\widetilde\PotN_{\RotRottS,\gat}^{k},\\
\id_{\H{k+1,k}{\S,\gat}(\RotRott,\om)}
&=\widetilde\PotQ_{\RotRottS,\gat}^{k+1,k,1}
+\widetilde\PotN_{\RotRottS,\gat}^{k+1,k}.
\end{align*}
\end{theo}

\begin{cor}[bounded regular kernel decompositions]
\label{cor:regdecokernelhigherorder}
For $k\geq0$ the bounded regular kernel decompositions
\begin{align*}
\H{k}{\S,\gat,0}(\Div,\om)
&=\H{k+1}{\S,\gat,0}(\Div,\om)
+\RotRott\H{k+2}{\S,\gat}(\om),\\
\H{k}{\S,\gat,0}(\RotRott,\om)
&=\H{k+2}{\S,\gat,0}(\RotRott,\om)
+\symGrad\H{k+1}{\gat}(\om)
\end{align*}
hold.
\end{cor}

\begin{rem}[bounded regular decompositions from bounded regular potentials]
\label{rem:regdecohigherorder}
It holds 
\begin{align*}
\Div\widetilde\PotQ_{\DivS,\gat}^{k,1}
=\Div\PotQ_{\DivS,\gat}^{k,1}
&=\DivSgatk
\intertext{and hence $\H{k}{\S,\gat,0}(\Div,\om)$ is invariant under 
$\PotQ_{\DivS,\gat}^{k,1}$ and $\widetilde\PotQ_{\DivS,\gat}^{k,1}$.
Analogously,}
\RotRott\widetilde\PotQ_{\RotRottS,\gat}^{k,1}
=\RotRott\PotQ_{\RotRottS,\gat}^{k,1}
&=\RotRottSgatk,\\
\RotRott\widetilde\PotQ_{\RotRottS,\gat}^{k+1,k,1}
=\RotRott\PotQ_{\RotRottS,\gat}^{k+1,k,1}
&=\RotRottSgatkpok.
\end{align*}
Thus $\H{k}{\S,\gat,0}(\RotRott,\om)$ and $\H{k+1}{\S,\gat,0}(\RotRott,\om)$ 
are invariant under $\PotQ_{\RotRottS,\gat}^{k,1}$
and $\PotQ_{\RotRottS,\gat}^{k+1,k,1}$, respectively.
Moreover,
\begin{align*}
R(\widetilde\PotQ_{\DivS,\gat}^{k,1})
&=R(\PotP_{\DivS,\gat}^{k}),
&
\widetilde\PotQ_{\DivS,\gat}^{k,1}
&=\PotQ_{\DivS,\gat}^{k,1}(\DivSgatk)_{\bot}^{-1}\DivSgatk.
\intertext{Therefore, we have
$\widetilde\PotQ_{\DivS,\gat}^{k,1}|_{D((\DivSgatk)_{\bot})}=\PotQ_{\DivS,\gat}^{k,1}|_{D((\DivSgatk)_{\bot})}$
and thus $\widetilde\PotQ_{\DivS,\gat}^{k,1}$ may differ from $\PotQ_{\DivS,\gat}^{k,1}$
only on $N(\DivSgatk)=\H{k}{\S,\gat,0}(\Div,\om)$.
Analogously,}
R(\widetilde\PotQ_{\RotRottS,\gat}^{k,1})
&=R(\PotP_{\RotRottS,\gat}^{k}),
&
\widetilde\PotQ_{\RotRottS,\gat}^{k,1}
&=\PotQ_{\RotRottS,\gat}^{k,1}(\RotRottSgatk)_{\bot}^{-1}\RotRottSgatk,\\
R(\widetilde\PotQ_{\RotRottS,\gat}^{k+1,k,1})
&=R(\PotP_{\RotRottS,\gat}^{k+1,k}),
&
\widetilde\PotQ_{\RotRottS,\gat}^{k+1,k,1}
&=\PotQ_{\RotRottS,\gat}^{k+1,k,1}(\RotRottSgatkpok)_{\bot}^{-1}\RotRottSgatkpok.
\end{align*}
Hence
\begin{align*}
\widetilde\PotQ_{\RotRottS,\gat}^{k,1}|_{D((\RotRottSgatk)_{\bot})}
&=\PotQ_{\RotRottS,\gat}^{k,1}|_{D((\RotRottSgatk)_{\bot})},\\
\widetilde\PotQ_{\RotRottS,\gat}^{k+1,k,1}|_{D((\RotRottSgatkpok)_{\bot})}
&=\PotQ_{\RotRottS,\gat}^{k+1,k,1}|_{D((\RotRottSgatkpok)_{\bot})}
\end{align*}
and thus $\widetilde\PotQ_{\RotRottS,\gat}^{k,1}$
and $\widetilde\PotQ_{\RotRottS,\gat}^{k+1,k,1}$ may differ from 
$\PotQ_{\RotRottS,\gat}^{k,1}$ and $\PotQ_{\RotRottS,\gat}^{k+1,k,1}$
only on the kernels $N(\RotRottSgatk)=\H{k}{\S,\gat,0}(\RotRott,\om)$
and $N(\RotRottSgatkpok)=\H{k+1}{\S,\gat,0}(\RotRott,\om)$, respectively.
\end{rem}

\begin{rem}[projections]
\label{rem:regdecohigherorderproj}
Recall Theorem \ref{theo:regdecohigherorder}, e.g.,
for $\RotRottSgatk$
$$\H{k}{\S,\gat}(\RotRott,\om)
=R(\widetilde\PotQ_{\RotRottS,\gat}^{k,1})
\dotplus R(\widetilde\PotN_{\RotRottS,\gat}^{k}).$$
\begin{itemize}
\item[\bf(i)]
$\widetilde\PotQ_{\RotRottS,\gat}^{k,1}$,
$\widetilde\PotN_{\RotRottS,\gat}^{k}=1-\widetilde\PotQ_{\RotRottS,\gat}^{k,1}$
are projections.
\item[\bf(i')] 
$\widetilde\PotQ_{\RotRottS,\gat}^{k,1}\widetilde\PotN_{\RotRottS,\gat}^{k}
=\widetilde\PotN_{\RotRottS,\gat}^{k}\widetilde\PotQ_{\RotRottS,\gat}^{k,1}=0$.
\item[\bf(ii)]
For $I_{\pm}:=\widetilde\PotQ_{\RotRottS,\gat}^{k,1}\pm\widetilde\PotN_{\RotRottS,\gat}^{k}$
it holds $I_{+}=I_{-}^{2}=\id_{\H{k}{\S,\gat}(\RotRott,\om)}$.
Therefore, $I_{+}$, $I_{-}^{2}$, as well as
$I_{-}=2\widetilde\PotQ_{\RotRottS,\gat}^{k,1}-\id_{\H{k}{\S,\gat}(\RotRott,\om)}$
are topological isomorphisms on $\H{k}{\S,\gat}(\RotRott,\om)$.
\item[\bf(iii)]
There exists $c>0$ such that for all $S\in\H{k}{\S,\gat}(\RotRott,\om)$
\begin{align*}
c\norm{\widetilde\PotQ_{\RotRottS,\gat}^{k,1}S}_{\H{k+2}{\S}(\om)}
&\leq\norm{\RotRott S}_{\H{k}{\S}(\om)}
\leq\norm{S}_{\H{k}{\S}(\RotRott,\om)},\\
\norm{\widetilde\PotN_{\RotRottS,\gat}^{k}S}_{\H{k}{\S}(\om)}
&\leq\norm{S}_{\H{k}{\S}(\om)}
+\norm{\widetilde\PotQ_{\RotRottS,\gat}^{k,1}S}_{\H{k}{\S}(\om)}.
\end{align*}
\item[\bf(iii')]
For $S\in\H{k}{\S,\gat,0}(\RotRott,\om)$ we have $\widetilde\PotQ_{\RotRottS,\gat}^{k,1}S=0$
and $\widetilde\PotN_{\RotRottS,\gat}^{k}S=S$.
In particular, $\widetilde\PotN_{\RotRottS,\gat}^{k}$ is onto.
\end{itemize}
Similar results to (i)-(iii') hold for $\DivSgatk$ and $\RotRottSgatkpok$ as well.
In particular, $\widetilde\PotQ_{\DivS,\gat}^{k,1}$,
$\widetilde\PotN_{\DivS,\gat}^{k}$, and
$\widetilde\PotQ_{\RotRottS,\gat}^{k+1,k,1}$,
$\widetilde\PotN_{\RotRottS,\gat}^{k+1,k}$
are projections and there exists $c>0$ such that for all 
$T\in\H{k}{\S,\gat}(\Div,\om)$ and all $S\in\H{k+1,k}{\S,\gat}(\RotRott,\om)$
\begin{align*}
\norm{\widetilde\PotQ_{\DivS,\gat}^{k,1}T}_{\H{k+1}{\S}(\om)}
&\leq c\norm{\Div T}_{\H{k}{}(\om)},
&
\norm{\widetilde\PotQ_{\RotRottS,\gat}^{k+1,k,1}S}_{\H{k+2}{\S}(\om)}
&\leq c\norm{\RotRott S}_{\H{k}{\S}(\om)}.
\end{align*}
\end{rem}

Corollary \ref{cor:regdecokernelhigherorder} shows:

\begin{cor}[bounded regular higher order kernel decompositions]
\label{cor:regdecokernelhigherorder2}
For $k,\ell\geq0$ the bounded regular kernel decompositions
\begin{align*}
N(\DivSgatk)
=\H{k}{\S,\gat,0}(\Div,\om)
&=\H{\ell}{\S,\gat,0}(\Div,\om)
+\RotRott\H{k+2}{\S,\gat}(\om),\\
N(\RotRottSgatk)
=\H{k}{\S,\gat,0}(\RotRott,\om)
&=\H{\ell}{\S,\gat,0}(\RotRott,\om)
+\symGrad\H{k+1}{\gat}(\om)
\end{align*}
hold. In particular, for $k=0$ and all $\ell\geq0$
\begin{align*}
N(\DivSgat)
=\H{}{\S,\gat,0}(\Div,\om)
&=\H{\ell}{\S,\gat,0}(\Div,\om)
+\RotRott\H{2}{\S,\gat}(\om),\\
N(\RotRottSgat)
=\H{}{\S,\gat,0}(\RotRott,\om)
&=\H{\ell}{\S,\gat,0}(\RotRott,\om)
+\symGrad\H{1}{\gat}(\om).
\end{align*}
\end{cor}

\subsection{Dirichlet/Neumann Fields}
\label{sec:dirneuB}%

From Theorem \ref{theo:miniFATzeroorder} (iv) we recall the 
orthonormal Helmholtz type decompositions (for $\mu=1$)
\begin{align}
\begin{aligned}
\label{helmcoho1}
\L{2}{\S,\eps}(\om)
&=R(\symGradgat)
\oplus_{\L{2}{\S,\eps}(\om)}N(\DivSgan\eps)\\
&=N(\RotRottSgat)
\oplus_{\L{2}{\S,\eps}(\om)}R(\eps^{-1}\RotRottSgan)\\
&=R(\symGradgat)
\oplus_{\L{2}{\S,\eps}(\om)}\Harm{}{\S,\gat,\gan,\eps}(\om)
\oplus_{\L{2}{\S,\eps}(\om)}R(\eps^{-1}\RotRottSgan),\\
N(\RotRottSgat)
&=R(\symGradgat)
\oplus_{\L{2}{\S,\eps}(\om)}\Harm{}{\S,\gat,\gan,\eps}(\om),\\
N(\DivSgan\eps)
&=\Harm{}{\S,\gat,\gan,\eps}(\om)
\oplus_{\L{2}{\S,\eps}(\om)}R(\eps^{-1}\RotRottSgan).
\end{aligned}
\end{align}
Let us denote the $\L{2}{\S,\eps}(\om)$-orthonormal projector onto 
$N(\DivSgan\eps)$ and $N(\RotRottSgat)$ by
$$\pi_{\Div}:\L{2}{\S,\eps}(\om)\to N(\DivSgan\eps),\qquad
\pi_{\RotRott}:\L{2}{\S,\eps}(\om)\to N(\RotRottSgat),$$
respectively. Then
\begin{align*}
\pi_{\Div}|_{N(\RotRottSgat)}:N(\RotRottSgat)&\to\Harm{}{\S,\gat,\gan,\eps}(\om),\\
\pi_{\RotRott}|_{N(\DivSgan\eps)}:N(\DivSgan\eps)&\to\Harm{}{\S,\gat,\gan,\eps}(\om)
\end{align*}
are onto. Moreover, 
\begin{align*}
\pi_{\Div}|_{R(\symGradgat)}&=0,
&
\pi_{\RotRott}|_{R(\eps^{-1}\RotRottSgan)}&=0,\\
\pi_{\Div}|_{\Harm{}{\S,\gat,\gan,\eps}(\om)}&=\id_{\Harm{}{\S,\gat,\gan,\eps}(\om)},
&
\pi_{\RotRott}|_{\Harm{}{\S,\gat,\gan,\eps}(\om)}&=\id_{\Harm{}{\S,\gat,\gan,\eps}(\om)}.
\end{align*}
Therefore, by Corollary \ref{cor:regdecokernelhigherorder2} and for all $\ell\geq0$
\begin{align*}
\Harm{}{\S,\gat,\gan,\eps}(\om)
&=\pi_{\Div}N(\RotRottSgat)
=\pi_{\Div}\H{\ell}{\S,\gat,0}(\RotRott,\om),\\
\Harm{}{\S,\gat,\gan,\eps}(\om)
&=\pi_{\RotRott}N(\DivSgan\eps)
=\pi_{\RotRott}\eps^{-1}\H{\ell}{\S,\gan,0}(\Div,\om),
\end{align*}
where we have used $N(\DivSgan\eps)=\eps^{-1}\H{}{\S,\gan,0}(\Div,\om)$.
Hence with
$$\H{\infty}{\S,\gat,0}(\RotRott,\om):=\bigcap_{k\geq0}\H{k}{\S,\gat,0}(\RotRott,\om),\qquad
\H{\infty}{\S,\gan,0}(\Div,\om):=\bigcap_{k\geq0}\H{k}{\S,\gan,0}(\Div,\om)$$
we have the following result:

\begin{theo}[smooth pre-bases of Dirichlet/Neumann fields]
\label{theo:cohomologyinfty}
Let $d_{\om,\gat}:=\dim\Harm{}{\S,\gat,\gan,\eps}(\om)$. Then 
$$\pi_{\Div}\H{\infty}{\S,\gat,0}(\RotRott,\om)
=\Harm{}{\S,\gat,\gan,\eps}(\om)
=\pi_{\RotRott}\eps^{-1}\H{\infty}{\S,\gan,0}(\Div,\om).$$
Moreover, there exists a smooth $\RotRott$-\emph{pre-basis}
and a smooth $\Div$-\emph{pre-basis} of $\Harm{}{\S,\gat,\gan,\eps}(\om)$, i.e.,
there are linear independent smooth fields 
\begin{align*}
\B{\RotRott}{\S,\gat}(\om)
&:=\{B_{\S,\gat,\ell}^{\RotRott}\}_{\ell=1}^{d_{\om,\gat}}
\subset\H{\infty}{\S,\gat,0}(\RotRott,\om),\\
\B{\Div}{\S,\gan}(\om)
&:=\{B_{\S,\gan,\ell}^{\Div}\}_{\ell=1}^{d_{\om,\gat}}
\subset\H{\infty}{\S,\gan,0}(\Div,\om),
\end{align*}
such that $\pi_{\Div}\B{\RotRott}{\S,\gat}(\om)$ and $\pi_{\RotRott}\eps^{-1}\B{\Div}{\S,\gan}(\om)$
are both bases of $\Harm{}{\S,\gat,\gan,\eps}(\om)$. In particular,
$$\Lin\pi_{\Div}\B{\RotRott}{\S,\gat}(\om)
=\Harm{}{\S,\gat,\gan,\eps}(\om)
=\Lin\pi_{\RotRott}\eps^{-1}\B{\Div}{\S,\gan}(\om).$$
\end{theo}

Note that $(1-\pi_{\Div})$ and $(1-\pi_{\RotRott})$ are the $\L{2}{\S,\eps}(\om)$-orthonormal projectors onto 
the ranges $R(\symGradgat)$ and $R(\eps^{-1}\RotRottSgan)$, respectively, i.e.,
\begin{align*}
(1-\pi_{\Div}):\L{2}{\S,\eps}(\om)&\to R(\symGradgat),
&
(1-\pi_{\RotRott}):\L{2}{\S,\eps}(\om)&\to R(\eps^{-1}\RotRottSgan).
\end{align*}
By \eqref{helmcoho1}, Theorem \ref{theo:rangeselawithoutbdpot},
and Theorem \ref{theo:cohomologyinfty} we have, e.g.,
\begin{align}
\begin{aligned}
\label{helmcoho2}
\H{}{\S,\gat,0}(\RotRott,\om)
&=R(\symGradgat)
\oplus_{\L{2}{\S,\eps}(\om)}
\Harm{}{\S,\gat,\gan,\eps}(\om)\\
&=R(\symGradgat)
\oplus_{\L{2}{\S,\eps}(\om)}
\Lin\pi_{\Div}\B{\RotRott}{\S,\gat}(\om)\\
&=R(\symGradgat)
+(\pi_{\Div}-1)\Lin\B{\RotRott}{\S,\gat}(\om)
+\Lin\B{\RotRott}{\S,\gat}(\om)\\
&=R(\symGradgat)
+\Lin\B{\RotRott}{\S,\gat}(\om),\\
\H{k}{\S,\gat,0}(\RotRott,\om)
&=R(\symGradgat)\cap\H{k}{\S,\gat,0}(\RotRott,\om)
+\Lin\B{\RotRott}{\S,\gat}(\om),\\
&=R(\symGradgatk)
+\Lin\B{\RotRott}{\S,\gat}(\om).
\end{aligned}
\end{align}

Similarly, we obtain a decomposition of 
$\H{k}{\S,\gan,0}(\Div,\om)$ using $\B{\Div}{\S,\gan}(\om)$. We conclude:

\begin{theo}[bounded regular direct decompositions]
\label{theo:highorderregdecoinfty}
Let $k\geq0$. 
Then the bounded regular direct decompositions
\begin{align*}
\H{k}{\S,\gat}(\RotRott,\om)
&=R(\widetilde\PotQ_{\RotRottS,\gat}^{k,1})
\dotplus\H{k}{\S,\gat,0}(\RotRott,\om),\\
\H{k+1,k}{\S,\gat}(\RotRott,\om)
&=R(\widetilde\PotQ_{\RotRottS,\gat}^{k+1,k,1})
\dotplus\H{k+1}{\S,\gat,0}(\RotRott,\om),\\
\H{k}{\S,\gat,0}(\RotRott,\om)
&=\symGrad\H{k+1}{\gat}(\om)
\dotplus\Lin\B{\RotRott}{\S,\gat}(\om),\\
\H{k}{\S,\gan}(\Div,\om)
&=R(\widetilde\PotQ_{\DivS,\gan}^{k,1})
\dotplus\H{k}{\S,\gan,0}(\Div,\om),\\
\H{k}{\S,\gan,0}(\Div,\om)
&=\RotRott\H{k+2}{\S,\gan}(\om)
\dotplus\Lin\B{\Div}{\S,\gan}(\om)
\end{align*}
hold. Note that 
$R(\widetilde\PotQ_{\RotRottS,\gat}^{k,1}),R(\widetilde\PotQ_{\RotRottS,\gat}^{k+1,k,1})\subset\H{k+2}{\S,\gat}(\om)$
and $R(\widetilde\PotQ_{\DivS,\gan}^{k,1})\subset\H{k+1}{\S,\gan}(\om)$.
\end{theo}

\begin{rem}[bounded regular direct decompositions]
\label{rem:highorderregdecoinfty}
In particular, for $k=0$
\begin{align*}
\H{}{\S,\gat}(\RotRott,\om)
&=R(\widetilde\PotQ_{\RotRottS,\gat}^{0,1})
\dotplus\H{}{\S,\gat,0}(\RotRott,\om),\\
\H{}{\S,\gat,0}(\RotRott,\om)
&=\symGrad\H{1}{\gat}(\om)
\dotplus\Lin\B{\RotRott}{\S,\gat}(\om)\\
&=\symGrad\H{1}{\gat}(\om)
\oplus_{\L{2}{\S,\eps}(\om)}
\Harm{}{\S,\gat,\gan,\eps}(\om),\\
\H{}{\S,\gan}(\Div,\om)
&=R(\widetilde\PotQ_{\DivS,\gan}^{0,1})
\dotplus\H{}{\S,\gan,0}(\Div,\om),\\
\eps^{-1}\H{}{\S,\gan,0}(\Div,\om)
&=\eps^{-1}\RotRott\H{2}{\S,\gan}(\om)
\dotplus\eps^{-1}\Lin\B{\Div}{\S,\gan}(\om)\\
&=\eps^{-1}\RotRott\H{2}{\S,\gan}(\om)
\oplus_{\L{2}{\S,\eps}(\om)}
\Harm{}{\S,\gat,\gan,\eps}(\om)
\intertext{and}
\L{2}{\S,\eps}(\om)
&=\H{}{\S,\gat,0}(\RotRott,\om)
\oplus_{\L{2}{\S,\eps}(\om)}
\eps^{-1}\RotRott\H{2}{\S,\gan}(\om)\\
&=\symGrad\H{1}{\gat}(\om)
\oplus_{\L{2}{\S,\eps}(\om)}
\eps^{-1}\H{}{\S,\gan,0}(\Div,\om).
\end{align*}
\end{rem}

\begin{proof}[Proof of Theorem \ref{theo:highorderregdecoinfty}.]
Theorem \ref{theo:regdecohigherorder} and \eqref{helmcoho2} show
\begin{align*}
\H{k}{\S,\gat}(\RotRott,\om)
&=R(\widetilde\PotQ_{\RotRottS,\gat}^{k,1})
\dotplus\H{k}{\S,\gat,0}(\RotRott,\om),\\
\H{k+1,k}{\S,\gat}(\RotRott,\om)
&=R(\widetilde\PotQ_{\RotRottS,\gat}^{k+1,k,1})
\dotplus\H{k+1}{\S,\gat,0}(\RotRott,\om),\\
\H{k}{\S,\gat,0}(\RotRott,\om)
&=\symGrad\H{k+1}{\gat}(\om)
+\Lin\B{\RotRott}{\S,\gat}(\om).
\end{align*}
To prove the directness, let
$$\sum_{\ell=1}^{d_{\om,\gat}}\lambda_{\ell}B_{\S,\gat,\ell}^{\RotRott}
\in\symGrad\H{k+1}{\gat}(\om)
\cap\B{\RotRott}{\S,\gat}(\om).$$
Then $0=\sum_{\ell}\lambda_{\ell}\pi_{\Div}B_{\S,\gat,\ell}^{\RotRott}
\in\Lin\pi_{\Div}B_{\S,\gat,\ell}^{\RotRott}$
and hence $\lambda_{\ell}=0$ for all $\ell$
as $\pi_{\Div}B_{\S,\gat,\ell}^{\RotRott}$ is a basis of 
$\Harm{}{\S,\gat,\gan,\eps}(\om)$
by Theorem \ref{theo:cohomologyinfty}.
Concerning the boundedness of the decompositions, let 
$$\H{k}{\S,\gat,0}(\RotRott,\om)\ni S=\symGrad v+B,\qquad
v\in\H{k+1}{\gat}(\om),\quad
B\in\Lin\B{\RotRott}{\S,\gat}(\om).$$
By Theorem \ref{theo:regpothigherorder}
$\symGrad v\in R(\symGradgatk)$ and 
$u:=\PotP_{\symGrad,\gat}^{k}\symGrad v\in\H{k+1}{\gat}(\om)$ solves
$\symGrad u=\symGrad v$ with 
$\norm{u}_{\H{k+1}{}(\om)}\leq c\norm{\symGrad v}_{\H{k}{\S}(\om)}$.
Therefore, 
$$\norm{u}_{\H{k+1}{}(\om)}
+\norm{B}_{\H{k}{\S}(\om)}
\leq c\big(\norm{\symGrad v}_{\H{k}{\S}(\om)}
+\norm{B}_{\H{k}{\S}(\om)}\big)
\leq c\big(\norm{S}_{\H{k}{\S}(\om)}
+\norm{B}_{\H{k}{\S}(\om)}\big).$$
Note that the mapping
$$I_{\pi,\Div}:\Lin\B{\RotRott}{\S,\gat}(\om)\to
\Lin\pi_{\Div}\B{\RotRott}{\S,\gat}(\om)
=\Harm{}{\S,\gat,\gan,\eps}(\om);\quad
B_{\S,\gat,\ell}^{\RotRott}\mapsto\pi_{\Div}B_{\S,\gat,\ell}^{\RotRott}$$
is a topological isomorphism 
(between finite dimensional spaces and with arbitrary norms). Thus 
$$\norm{B}_{\H{k}{\S}(\om)}
\leq c\norm{B}_{\L{2}{\S}(\om)}
\leq c\norm{\pi_{\Div}B}_{\L{2}{\S}(\om)}
=c\norm{\pi_{\Div}S}_{\L{2}{\S}(\om)}
\leq c\norm{S}_{\L{2}{\S}(\om)}
\leq c\norm{S}_{\H{k}{\S}(\om)}.$$
Finally, we see 
$S=\symGrad u+B
\in\symGrad\H{k+1}{\gat}(\om)
\dotplus\Lin\B{\RotRott}{\S,\gat}(\om)$
and 
$$\norm{u}_{\H{k+1}{}(\om)}
+\norm{B}_{\H{k}{\S}(\om)}
\leq c\norm{S}_{\H{k}{\S}(\om)}.$$
The other assertions for $\Div$ follow analogously.
\end{proof}

\begin{rem}[bounded regular direct decompositions]
\label{rem:highorderregdecoinfty2}
By Theorem \ref{theo:highorderregdecoinfty} we have, e.g.,
\begin{align*}
\H{k}{\S,\gat}(\RotRott,\om)
&=R(\widetilde\PotQ_{\RotRottS,\gat}^{k,1})
\dotplus\Lin\B{\RotRott}{\S,\gat}(\om)
\dotplus\symGrad\H{k+1}{\gat}(\om)\\
&=\H{k+2}{\S,\gat}(\om)
+\symGrad\H{k+1}{\gat}(\om)
\end{align*}
with bounded linear regular direct decomposition operators
\begin{align*}
\widehat\PotQ_{\RotRottS,\gat}^{k,1}:
\H{k}{\S,\gat}(\RotRott,\om)&\to R(\widetilde\PotQ_{\RotRottS,\gat}^{k,1})
\subset\H{k+2}{\S,\gat}(\om),\\
\widehat\PotQ_{\RotRottS,\gat}^{k,\infty}:
\H{k}{\S,\gat}(\RotRott,\om)&\to\Lin\B{\RotRott}{\S,\gat}(\om)
\subset\H{\infty}{\S,\gat,0}(\RotRott,\om)
\subset\H{k+2}{\S,\gat}(\om),\\
\widehat\PotQ_{\RotRottS,\gat}^{k,0}:
\H{k}{\S,\gat}(\RotRott,\om)&\to
\H{k+1}{\gat}(\om)
\end{align*}
satisfying 
$\widehat\PotQ_{\RotRottS,\gat}^{k,1}
+\widehat\PotQ_{\RotRottS,\gat}^{k,\infty}
+\symGrad\widehat\PotQ_{\RotRottS,\gat}^{k,0}
=\id_{\H{k}{\S,\gat}(\RotRott,\om)}$.

A closer inspection of the latter proof 
allows for a more precise description of these bounded decomposition operators.
For this, let $S\in\H{k}{\S,\gat}(\RotRott,\om)$. 
According to Theorem \ref{theo:regdecohigherorder} and 
Remark \ref{rem:regdecohigherorderproj} we decompose 
$$S=S_{R}+S_{N}\in 
R(\widetilde\PotQ_{\RotRottS,\gat}^{k,1})
\dotplus 
R(\widetilde\PotN_{\RotRottS,\gat}^{k})$$
with 
$R(\widetilde\PotN_{\RotRottS,\gat}^{k})
=\H{k}{\S,\gat,0}(\RotRott,\om)
=N(\RotRottSgatk)$
as well as $S_{R}=\widetilde\PotQ_{\RotRottS,\gat}^{k,1}S$
and $S_{N}=\widetilde\PotN_{\RotRottS,\gat}^{k}S$.
By Theorem \ref{theo:highorderregdecoinfty} we further decompose
\begin{align*}
\H{k}{\S,\gat,0}(\RotRott,\om)\ni 
S_{N}
=\symGrad u+B
\in\symGrad\H{k+1}{\gat}(\om)
\dotplus\Lin\B{\RotRott}{\S,\gat}(\om).
\end{align*}
Then 
$\pi_{\Div}S_{N}
=\pi_{\Div}B\in\Harm{}{\S,\gat,\gan,\eps}(\om)$
and thus 
$B=I_{\pi,\Div}^{-1}\pi_{\Div}S_{N}\in\Lin\B{\RotRott}{\S,\gat}(\om)$.
Therefore, 
$u=\PotP_{\symGrad,\gat}^{k}\symGrad u
=\PotP_{\symGrad,\gat}^{k}(S_{N}-B)
=\PotP_{\symGrad,\gat}^{k}(1-I_{\pi,\Div}^{-1}\pi_{\Div})S_{N}$.
Finally we see
\begin{align*}
\widehat\PotQ_{\RotRottS,\gat}^{k,1}
&=\widetilde\PotQ_{\RotRottS,\gat}^{k,1}
=\PotP_{\RotRottS,\gat}^{k}\RotRottSgatk
=\PotQ_{\RotRottS,\gat}^{k,1}(\RotRottSgatk)_{\bot}^{-1}\RotRottSgatk,\\
\widehat\PotQ_{\RotRottS,\gat}^{k,\infty}
&=I_{\pi,\Div}^{-1}\pi_{\Div}\widetilde\PotN_{\RotRottS,\gat}^{k},\\
\widehat\PotQ_{\RotRottS,\gat}^{k,0}
&=\PotP_{\symGrad,\gat}^{k}(1-I_{\pi,\Div}^{-1}\pi_{\Div})\widetilde\PotN_{\RotRottS,\gat}^{k}
\end{align*}
with 
$\widetilde\PotN_{\RotRottS,\gat}^{k}
=1-\widetilde\PotQ_{\RotRottS,\gat}^{k,1}$.
Analogously, we have
\begin{align*}
\H{k+1,k}{\S,\gat}(\RotRott,\om)
&=R(\widetilde\PotQ_{\RotRottS,\gat}^{k+1,k,1})
\dotplus\Lin\B{\RotRott}{\S,\gat}(\om)
\dotplus\symGrad\H{k+2}{\gat}(\om)\\
&=\H{k+2}{\S,\gat}(\om)
+\symGrad\H{k+2}{\gat}(\om),\\
\H{k}{\S,\gan}(\Div,\om)
&=R(\widetilde\PotQ_{\DivS,\gan}^{k,1})
\dotplus\Lin\B{\Div}{\S,\gan}(\om)
\dotplus\RotRott\H{k+2}{\S,\gan}(\om)\\
&=\H{k+1}{\S,\gan}(\om)
+\RotRott\H{k+2}{\S,\gan}(\om)
\end{align*}
with bounded linear regular direct decomposition operators
\begin{align*}
\widehat\PotQ_{\RotRottS,\gat}^{k+1,k,1}:
\H{k+1,k}{\S,\gat}(\RotRott,\om)&\to R(\widetilde\PotQ_{\RotRottS,\gat}^{k+1,k,1})
\subset\H{k+2}{\S,\gat}(\om),\\
\widehat\PotQ_{\RotRottS,\gat}^{k+1,k,\infty}:
\H{k+1,k}{\S,\gat}(\RotRott,\om)&\to\Lin\B{\RotRott}{\S,\gat}(\om)
\subset\H{\infty}{\S,\gat,0}(\RotRott,\om)
\subset\H{k+2}{\S,\gat}(\om),\\
\widehat\PotQ_{\RotRottS,\gat}^{k+1,k,0}:
\H{k+1,k}{\S,\gat}(\RotRott,\om)&\to
\H{k+2}{\gat}(\om),\\
\widehat\PotQ_{\DivS,\gan}^{k,1}:
\H{k}{\S,\gan}(\Div,\om)&\to R(\widetilde\PotQ_{\DivS,\gan}^{k,1})
\subset\H{k+1}{\S,\gan}(\om),\\
\widehat\PotQ_{\DivS,\gan}^{k,\infty}:
\H{k}{\S,\gan}(\Div,\om)&\to\Lin\B{\Div}{\S,\gan}(\om)
\subset\H{\infty}{\S,\gan,0}(\Div,\om)
\subset\H{k+1}{\S,\gan}(\om),\\
\widehat\PotQ_{\DivS,\gan}^{k,0}:
\H{k}{\S,\gan}(\Div,\om)&\to
\H{k+2}{\S,\gan}(\om)
\end{align*}
satisfying 
\begin{align*}
\widehat\PotQ_{\RotRottS,\gat}^{k+1,k,1}
+\widehat\PotQ_{\RotRottS,\gat}^{k+1,k,\infty}
+\symGrad\widehat\PotQ_{\RotRottS,\gat}^{k+1,k,0}
&=\id_{\H{k+1,k}{\S,\gat}(\RotRott,\om)},\\
\widehat\PotQ_{\DivS,\gan}^{k,1}
+\widehat\PotQ_{\DivS,\gan}^{k,\infty}
+\RotRott\widehat\PotQ_{\DivS,\gan}^{k,0}
&=\id_{\H{k}{\S,\gan}(\Div,\om)}
\end{align*}
and 
\begin{align*}
\widehat\PotQ_{\RotRottS,\gat}^{k+1,k,1}
&=\widetilde\PotQ_{\RotRottS,\gat}^{k+1,k,1}
=\PotP_{\RotRottS,\gat}^{k+1,k}\RotRottSgatkpok,\\
\widehat\PotQ_{\RotRottS,\gat}^{k+1,k,\infty}
&=I_{\pi,\Div}^{-1}\pi_{\Div}\widetilde\PotN_{\RotRottS,\gat}^{k+1,k},\\
\widehat\PotQ_{\RotRottS,\gat}^{k+1,k,0}
&=\PotP_{\symGrad,\gat}^{k+1}(1-I_{\pi,\Div}^{-1}\pi_{\Div})\widetilde\PotN_{\RotRottS,\gat}^{k+1,k},\\
\widehat\PotQ_{\DivS,\gan}^{k,1}
&=\widetilde\PotQ_{\DivS,\gan}^{k,1}
=\PotP_{\DivS,\gan}^{k}\DivSgank,\\
\widehat\PotQ_{\DivS,\gan}^{k,\infty}
&=I_{\pi,\RotRott}^{-1}\pi_{\RotRott}\widetilde\PotN_{\DivS,\gan}^{k},\\
\widehat\PotQ_{\DivS,\gan}^{k,0}
&=\PotP_{\RotRottS,\gan}^{k}(1-I_{\pi,\RotRott}^{-1}\pi_{\RotRott})\widetilde\PotN_{\DivS,\gan}^{k}
\end{align*}
with 
\begin{align*}
\widetilde\PotN_{\RotRottS,\gat}^{k+1,k}
&=1-\widetilde\PotQ_{\RotRottS,\gat}^{k+1,k,1},
&
\widetilde\PotN_{\DivS,\gan}^{k}
&=1-\widehat\PotQ_{\DivS,\gan}^{k,1},\\
\PotP_{\RotRottS,\gat}^{k+1,k}
&=\PotQ_{\RotRottS,\gat}^{k+1,k,1}(\RotRottSgatkpok)_{\bot}^{-1},
&
\PotP_{\DivS,\gan}^{k}
&=\PotQ_{\DivS,\gan}^{k,1}(\DivSgank)_{\bot}^{-1},
\end{align*}
and
$$I_{\pi,\RotRott}:\Lin\B{\Div}{\S,\gan}(\om)\to
\Lin\pi_{\RotRott}\B{\Div}{\S,\gan}(\om)
=\Harm{}{\S,\gan,\gan,\eps}(\om);\quad
B_{\S,\gan,\ell}^{\Div}\mapsto\pi_{\RotRott}B_{\S,\gan,\ell}^{\Div}.$$
\end{rem}

Noting
\begin{align}
\label{Bortho}
R(\eps^{-1}\RotRottSgan)\bot_{\L{2}{\S,\eps}(\om)}\B{\RotRott}{\S,\gat}(\om),\qquad
R(\symGradgat)\bot_{\L{2}{\S}(\om)}\B{\Div}{\S,\gan}(\om)
\end{align}
we see:

\begin{theo}[alternative Dirichlet/Neumann projections]
\label{theo:Bcoho1}
It holds
\begin{align*}
\Harm{}{\S,\gat,\gan,\eps}(\om)\cap\B{\RotRott}{\S,\gat}(\om)^{\bot_{\L{2}{\S,\eps}(\om)}}
&=\{0\},\\
N(\DivSgan\eps)\cap\B{\RotRott}{\S,\gat}(\om)^{\bot_{\L{2}{\S,\eps}(\om)}}
&=R(\eps^{-1}\RotRottSgan),\\
\Harm{}{\S,\gat,\gan,\eps}(\om)\cap\B{\Div}{\S,\gan}(\om)^{\bot_{\L{2}{\S}(\om)}}
&=\{0\},\\
N(\RotRottSgat)\cap\B{\Div}{\S,\gan}(\om)^{\bot_{\L{2}{\S}(\om)}}
&=R(\symGradgat).
\end{align*}
Moreover, for all $k\geq0$
\begin{align*}
N(\DivSgank\eps)\cap\B{\RotRott}{\S,\gat}(\om)^{\bot_{\L{2}{\S,\eps}(\om)}}
&=R(\eps^{-1}\RotRottSgank)
=\eps^{-1}\RotRott\H{k+2}{\S,\gan}(\om),\\
N(\RotRottSgatk)\cap\B{\Div}{\S,\gan}(\om)^{\bot_{\L{2}{\S}(\om)}}
&=R(\symGradgatk)
=\symGrad\H{k+1}{\gat}(\om).
\end{align*}
\end{theo}

\begin{proof}
For $k=0$ and 
$S\in\Harm{}{\S,\gat,\gan,\eps}(\om)\cap\B{\RotRott}{\S,\gat}(\om)^{\bot_{\L{2}{\S,\eps}(\om)}}$
we have
\begin{align*}
0=\scp{S}{B_{\S,\gat,\ell}^{\RotRott}}_{\L{2}{\S,\eps}(\om)}
=\scp{\pi_{\Div}S}{B_{\S,\gat,\ell}^{\RotRott}}_{\L{2}{\S,\eps}(\om)}
=\scp{S}{\pi_{\Div}B_{\S,\gat,\ell}^{\RotRott}}_{\L{2}{\S,\eps}(\om)}
\end{align*}
and hence $S=0$ by Theorem \ref{theo:cohomologyinfty}.
Analogously, we see for 
$S\in\Harm{}{\S,\gat,\gan,\eps}(\om)\cap\B{\Div}{\S,\gan}(\om)^{\bot_{\L{2}{\S}(\om)}}$
\begin{align*}
0=\scp{S}{B_{\S,\gan,\ell}^{\Div}}_{\L{q}{}(\om)}
=\scp{\pi_{\RotRott}S}{\eps^{-1}B_{\S,\gan,\ell}^{\Div}}_{\L{2}{\S,\eps}(\om)}
=\scp{S}{\pi_{\RotRott}\eps^{-1}B_{\S,\gan,\ell}^{\Div}}_{\L{2}{\S,\eps}(\om)}
\end{align*}
and thus $S=0$ again by Theorem \ref{theo:cohomologyinfty}.
According to \eqref{helmcoho1} we can decompose
\begin{align*}
N(\DivSgan\eps)
&=R(\eps^{-1}\RotRottSgan)
\oplus_{\L{2}{\S,\eps}(\om)}
\Harm{}{\S,\gat,\gan,\eps}(\om),\\
N(\RotRottSgat)
&=R(\symGradgat)
\oplus_{\L{2}{\S,\eps}(\om)}
\Harm{}{\S,\gat,\gan,\eps}(\om),
\end{align*}
which shows by \eqref{Bortho} the other two assertions.
Let $k\geq0$. The case $k=0$ and Theorem \ref{theo:rangeselawithoutbdpot} show
\begin{align*}
N(\RotRottSgatk)\cap\B{\Div}{\S,\gan}(\om)^{\bot_{\L{2}{\S}(\om)}}
&=\H{k}{\S,\gat}(\om)\cap N(\RotRottSgat)\cap\B{\Div}{\S,\gan}(\om)^{\bot_{\L{2}{\S}(\om)}}\\
&=\H{k}{\S,\gat}(\om)\cap R(\symGradgat)\\
&=R(\symGradgatk)
=\symGrad\H{k+1}{\gat}(\om).
\end{align*}
Analogously,
\begin{align*}
N(\DivSgank\eps)\cap\B{\RotRott}{\S,\gat}(\om)^{\bot_{\L{2}{\S,\eps}(\om)}}
&=\eps^{-1}\H{k}{\S,\gan}(\om)\cap N(\DivSgan\eps)\cap\B{\RotRott}{\S,\gat}(\om)^{\bot_{\L{2}{\S,\eps}(\om)}}\\
&=\eps^{-1}\H{k}{\S,\gan}(\om)\cap R(\eps^{-1}\RotRottSgan)\\
&=R(\eps^{-1}\RotRottSgank)
=\eps^{-1}\RotRott\H{k+2}{\S,\gan}(\om),
\end{align*}
completing the proof.
\end{proof}

Theorem \ref{theo:highorderregdecoinfty} implies:

\begin{theo}[cohomology groups]
\label{theo:Bcoho2}
It holds
$$\frac{N(\RotRottSgatk)}{R(\symGradgatk)}
\cong\Lin\B{\RotRott}{\S,\gat}(\om)
\cong\Harm{}{\S,\gat,\gan,\eps}(\om)
\cong\Lin\B{\Div}{\S,\gan}(\om)
\cong\frac{N(\DivSgank)}{R(\RotRottSgank)}.$$
In particular, the dimensions of the cohomology groups
(Dirichlet/Neumann fields) are independent of $k$ and $\eps$ and it holds
\begin{align*}
d_{\om,\gat}
=\dim\big(N(\RotRottSgatk)/R(\symGradgatk)\big)
=\dim\big(N(\DivSgank)/R(\RotRottSgank)\big).
\end{align*}
\end{theo}


\bibliographystyle{plain} 
\bibliography{/Users/paule/GoogleDriveData/Tex/input/bibTex/psz}


\appendix

\section{Elementary Formulas}
\label{app:sec:formulas}

From \cite{PZ2016a,PZ2020a,PZ2020b} and \cite{PW2020a} we 
have the following collection of formulas related to the elasticity and 
the biharmonic complex.

\begin{lem}[{\cite[Lemma 12.10]{PW2020a}}]
\label{lem:PZformulalem}
Let $u$, $v$, $w$, and $S$ belong to $\C{\infty}{}(\reals^{3})$.
\begin{itemize}
\item
$(\spn v)\,w=v\times w=-(\spn w)\,v$ 
\quad and  
\quad
$(\spn v)(\spn^{-1}S)=-Sv$, if $\sym S=0$
\item
$\sym\spn v=0$
\quad and \quad
$\dev(u\id)=0$
\item
$\tr\Grad v=\div v$
\quad and \quad
$2\skw\Grad v=\spn\rot v$
\item
$\Div(u\id)=\grad u$
\quad and \quad
$\Rot(u\id)=-\spn\grad u$,\\
in particular, \quad $\rot\Div(u\id)=0$ 
\quad and \quad $\rot\spn^{-1}\Rot(u\id)=0$\\
and \quad $\sym\Rot(u\id)=0$
\item
$\Div\spn v=-\rot v$
\quad and \quad
$\Div\skw S=-\rot\spn^{-1}\skw S$,\\
in particular, \quad $\div\Div\skw S=0$
\item
$\Rot\spn v=(\div v)\id-(\Grad v)^{\top}$\\
and \quad $\Rot\skw S=(\div\spn^{-1}\skw S)\id-(\Grad\spn^{-1}\skw S)^{\top}$
\item
$\dev\Rot\spn v=-(\dev\Grad v)^{\top}$
\item
$-2\Rot\sym\Grad v=2\Rot\skw\Grad v=-(\Grad\rot v)^{\top}$
\item
$2\spn^{-1}\skw\Rot S=\Div S^{\top}-\grad\tr S=\Div\big(S-(\tr S)\id\big)^{\top}$,\\
in particular, \quad $\rot\Div S^{\top}=2\rot\spn^{-1}\skw\Rot S$\\
and \quad $2\skw\Rot S=\spn\Div S^{\top}$, if $\tr S=0$
\item
$\tr\Rot S=2\div\spn^{-1}\skw S$,
\quad in particular, \quad $\tr\Rot S=0$, if $\skw S=0$,\\
and \quad $\tr\Rot\sym S=0$ \quad and \quad $\tr\Rot\skw S=\tr\Rot S$
\item
$2(\Grad\spn^{-1}\skw S)^{\top}=(\tr\Rot\skw S)\id-2\Rot\skw S$
\item
$3\Div(\dev\Grad v)^{\top}=2\grad\div v$
\item
$2\Rot\sym\Grad v=-2\Rot\skw\Grad v=-\Rot\spn\rot v=(\Grad\rot v)^{\top}$
\item
$2\Div\sym\Rot S=-2\Div\skw\Rot S=\rot\Div S^{\top}$
\item
$\Rot(\Rot\sym S)^{\top}=\sym\Rot(\Rot S)^{\top}$
\item
$\Rot(\Rot\skw S)^{\top}=\skw\Rot(\Rot S)^{\top}$
\end{itemize}
All formulas extend to distributions as well.
\end{lem}


\vspace*{5mm}
\hrule
\vspace*{3mm}


\end{document}